\documentclass[oneside,a4paper,12pt,reqno]{amsart}
\usepackage[T1]{fontenc}
\usepackage[utf8]{inputenc}
\usepackage{graphicx}

\usepackage{amsmath,hyperref}
\usepackage{amsfonts}
\usepackage{amssymb}
\usepackage{amsthm}
\usepackage{array}
\usepackage{color}
\usepackage{enumerate}
\usepackage{tikz}
\usepackage{tikz-cd}

\newtheorem{theorem}{Theorem}[section]

\newtheorem{definition}[theorem]{Definition}
\newtheorem{proposition}[theorem]{Proposition}
\newtheorem{lemma}[theorem]{Lemma}
\newtheorem{corollary}[theorem]{Corollary}

\newtheorem{remark}[theorem]{Remark}

\numberwithin{equation}{section}

\newcommand{\N}{\mathbb{N}}
\newcommand{\Z}{\mathbb{Z}}

\newcommand{\R}{\mathbb{R}}
\newcommand{\C}{\mathbb{C}}

\newcommand{\Ebb}{\mathbb{E}}

\newcommand{\Lbb}{\mathbb{L}}

\newcommand{\Pbb}{\mathbb{P}}

\newcommand{\Bcal}{\mathcal{B}}

\newcommand{\Ecal}{\mathcal{E}}
\newcommand{\Fcal}{\mathcal{F}}

\newcommand{\Jcal}{\mathcal{J}}

\newcommand{\Lcal}{\mathcal{L}}
\newcommand{\Mcal}{\mathcal{M}}

\newcommand{\Pcal}{\mathcal{P}}

\newcommand{\norm}[2]{\left\| #1 \right\|_{#2}}
\newcommand{\dd}{\;{\rm d}}
\newcommand{\de}{{\rm d}}

\DeclareMathOperator{\Alt}{Alt}

\DeclareMathOperator{\Leb}{Leb}

\DeclareMathOperator{\Jac}{Jac}

\DeclareMathOperator{\Supp}{Supp}
\DeclareMathOperator{\Esup}{Esup}

\DeclareMathOperator{\Var}{Var}

\DeclareMathOperator{\mathm}{m}

\usepackage[top=3cm, bottom=2cm, left=1.5cm, right=1.5cm]{geometry}

\title{Probabilistic potential theory and induction of dynamical systems}
\author{Fran\c{c}oise P\`ene}
\address{Univ Brest, Universit\'e de Brest, IUF, Institut Universitaire de France, LMBA,
Laboratoire de Math\'ematiques de Bretagne Atlantique, UMR CNRS 6205,
6 avenue Le Gorgeu, 29238 Brest Cedex, France}
\email{francoise.pene@univ-brest.fr}
\author{Damien Thomine}
\address{D\'epartement de Math\'ematiques d'Orsay, Universit\'e Paris-Sud, 
UMR CNRS 8628, F-91405 Orsay Cedex, France}
\email{damien.thomine@math.u-psud.fr}

\date{September 2019, 9th}

\begin{document}

\begin{abstract}
In this article, we outline a version of a balayage formula in 
probabilistic potential theory adapted to measure-preserving dynamical systems. 
This balayage identity generalizes the property that induced maps 
preserve the restriction of the original invariant measure. 
As an application, we prove in some cases the invariance under 
induction of the Green-Kubo formula, as well as the invariance of 
a new degree $3$ invariant.
\end{abstract}

\maketitle

The central objects of the probabilistic theory of potential~\cite{Revuz:1975, Doob:1984} are the solutions of the Poisson equation:
\begin{equation*}
 (I-P) (f) 
 = g,
\end{equation*}
where $P$ is the transition kernel of a Markov chain and $g$ is fixed. Its solutions exhibit, 
in particular, an invariance under induction~\cite[Chapter~8.2]{Revuz:1975}. Given a subset $\Psi$ of the state space, if $P_\Psi$ is 
the transition kernel for the induced Markov chain, then one can deduce the solutions of 
the equation $(I-P_\Psi) f = g$ from those of the initial equation $(I-P) (f) = g$.
This invariance, in turn, is a powerful tool to compute $P_\Psi$, and from there hitting probabilities: 
if one is given a starting site and a number of targets, it is possible to compute the distribution 
of the first target hit by the Markov chain~\cite{Spitzer:1976}.

\smallskip

A number of physically or geometrically relevant dynamical systems, such as the Lorentz gas 
or the geodesic flow on abelian covers of hyperbolic manifolds, behave globally or locally 
like random walks. For instance, they satisfy global~\cite{Nagaev:1957} and local central limit theorems, invariance 
principles~\cite{Gouezel:2010}, large deviations~\cite{Young:1990}, etc. 
This raises the question of adapting the probabilistic potential theory to such systems. 
In a previous work~\cite{PeneThomine:2017}, the authors devised a method related to 
this theory to estimate the hitting probability of a single 
far away target for such systems. It relied on a stronger form of invariance under induction 
satisfied by Green-Kubo's bilinear form:
\begin{equation}
 \label{eq:GKIntro}
 \sigma^2_{GK} (A, \mathm, T; f,f) 
 := \int_A f^2 \dd \mathm + 2 \sum_{n \geq 1} \int_A f \cdot f\circ T^n \dd \mathm\, ,
\end{equation}
which appear is the limiting variance in the central limit theorem. 
While the method used in~\cite{PeneThomine:2017} does not extend to a larger 
number of targets, it suggests the possibility of applying potential theory to dynamical systems.

\smallskip

In this article, we show how to adapt the invariance under induction of the Poisson equation 
to general recurrent measure-preserving dynamical systems. Let $(A, \mathm, T)$ be a measure-preserving 
recurrent dynamical system, with $\Lcal$ the transfer operator associated to $(A,\mathm,T)$. 
Given $B\subset A$ and $T_B$ the first return map of $T$ to $B$, the system $(B, \mathm_{|B}, T_B)$ 
is recurrent; let $\Lcal_B$ be the associated transfer operator. 
Our first result shall be:

\begin{proposition}
\label{prop:InvarianceInduction}

Let $(A, \mathm, T)$ be measure-preserving, with $\mathm$ a recurrent $\sigma$-finite measure. 
Let $B \subset A$ be such that $0<\mathm(B) \leq +\infty$.

\smallskip

Let $p \in [1, \infty]$ and $f$, $g \in \Lbb^p (A, \mathm)$ be such that $g \equiv 0$ on $B^c$ and:
\begin{equation}
\label{eq:PoissonOriginal}
(I-\Lcal) f = g.
\end{equation}
Then:
\begin{equation}
\label{eq:PoissonInduit}
(I-\Lcal_B) f_{|B} = g_{|B}.
\end{equation}
\end{proposition}

This statement can be seen either as a generalization of the fact that, under these hypotheses, 
$(B, \mathm_{|B}, T_B)$ is measure-preserving, or as an 
application of a classical result from potential theory~\cite[Corollary~1.11]{Revuz:1975} 
to the Markov kernel $\Lcal$.

\smallskip

Proposition~\ref{prop:InvarianceInduction} can be used to explain the invariance by induction of Greenk-Kubo's formula~\eqref{eq:GKIntro}, 
which was noticed and leveraged in~\cite{PeneThomine:2017}, as well as the existence of higher-order invariants.

\smallskip

In the first part of this article, we present a self-contained proof of Proposition~\ref{prop:InvarianceInduction}, 
and investigate some general properties of the Poisson equation such as existence and uniqueness of its solutions, 
and adaptations to more general, dynamically relevant operators or functions. 
In Section~\ref{sec:DynamiqueMarkov} we make explicit the relationship between Proposition~\ref{prop:InvarianceInduction} and 
the known results in the theory of Markov chains, and how one can go from one setting 
to the other. The invariance under induction of Green-Kubo's formula and its relation 
with the properties of the solutions of the Poisson equation are discussed in Section~\ref{sec:GreenKubo}. 
The insights gained are applied in Section~\ref{sec:Degre3} to study a degree $3$ invariant. 
Finally, Section~\ref{sec:Distributions} discusses the distributional point of view 
on Green-Kubo's formula.

\section{Induction invariance and the transfer operator}
\label{sec:InvarianceInduction}

\subsection{General case}
\label{subsec:CasGeneral}

Let $A$ be a Polish space, $T : A \to A$ be measurable, and $\mathm$ a 
$\sigma$-finite\footnote{In this article, the space of $\sigma$-finite measures shall always include the space of finite measures.} 
$T$-invariant measure on $A$. Assume that $(A, \mathm, T)$ 
is recurrent. Let $B \subset A$ be measurable, with $0 < \mathm (B) < + \infty$. 
Let $\varphi_B:A\rightarrow \mathbb N^*\cup\{\infty\}$ be the first hitting time of $B$, defined by :
\begin{equation*}
\varphi_B (x) 
:= \inf \{n \geq 1: \ T^n (x) \in B\}.
\end{equation*}
As the system $(A, \mathm, T)$ is recurrent, $\varphi_B < + \infty$ almost everywhere on $B$. 
The induced transformation on $B$ is then defined as:
\begin{equation*}
T_B : \left\{
\begin{array}{lll}
B & \to & B \\
x & \mapsto & T^{\varphi_B(x)} (x)
\end{array}
\right. .
\end{equation*}
By the previous remark, $T_B$ is well-defined $\mathm$-almost everywhere on $B$.

\smallskip

The transfer operators $\Lcal$, $\Lcal_B$ on $\Lbb^p$ are defined as the duals of the Koopman (composition) operator on $\Lbb^q$, 
where $q$ is the conjugate of $p$ (i.e. $q$ is such that $\frac 1p+\frac 1q=1$):
\begin{align*}
\int_A \Lcal (f) \cdot g \dd \mathm 
& = \int_A f \cdot g \circ T \dd \mathm \quad \forall f \in \Lbb^p (A, \mathm), \quad \forall g \in \Lbb^q (A, \mathm), \\
\int_B \Lcal_B (f) \cdot g \dd \mathm 
& = \int_B f \cdot g \circ T_B \dd \mathm \quad \forall f \in \Lbb^p (B, \mathm), \quad \forall g \in \Lbb^q (B, \mathm).
\end{align*}
The Koopman operator acts by isometry on each $\Lbb^q$, so the transfer operators are weak contractions on each $\Lbb^p$. 
Let us turn to the proof of Proposition~\ref{prop:InvarianceInduction}.

\smallskip

A classical result in ergodic theory states that, given a measure-preserving recurrent dynamical system $(A, \mathm, T)$ 
and $B$ such that $0< \mathm(B) < +\infty$, the measure $\mathm_{|B}$ is $T_B$-invariant ~\cite[Proposition~1.5.3]{Aaronson:1997}. 
Proposition~\ref{prop:InvarianceInduction} generalises this result; see Corollary~\ref{cor:InvarianceMesure} for more details. 
Its proof follows the same line of thought, with more bookkeeping. 

\smallskip

In the proof of Proposition~\ref{prop:InvarianceInduction}, we also expend some effort to deal with the case $\mathm (B) = + \infty$. This is done by reinducing on 
an arbitrary $\widetilde{B} \subset B$ whose measure is finite, which shows that mass can't escape from $B$. 
The proof can be significantly simplified if one is only interested in the finite measure case, for instance 
for pedagogical purposes.

\begin{proof}[Proof of Proposition~\ref{prop:InvarianceInduction}]

Let $f$, $g$ be as in the proposition. 
To simplify notations, we write simply $\varphi$ for $\varphi_B$.
For all $n \geq 1$, let $A_n := \{\varphi = n\}$, 
then let $B_n := A_n \cap B$ and $C_n := A_n \cap B^c$. Denote by $q$ the conjugate of $p$.

\medskip
\textbf{\textsc{Recursive formula}}
\smallskip

We claim that, for all $h \in \Lbb^q (A, \mathm)$, for all $n \geq 1$,
\begin{align}
\int_B f \cdot h \dd \mathm 
& = \int_B g \cdot h \dd \mathm + \sum_{k=1}^n \int_{B_k} f \cdot h \circ T^k \dd \mathm  + \int_{C_n} f \cdot h \circ T^n \dd \mathm \label{eq:InvarianceRecurrence} \\
& = \int_B g \cdot h \dd \mathm + \int_B \mathbf{1}_{\{\varphi \leq n\}} f \cdot h \circ T^k \dd \mathm  + \int_{C_n} f \cdot h \circ T^n \dd \mathm .\, \nonumber
\end{align}
Indeed, 
\begin{align*}
\int_B f \cdot h \dd \mathm 
& = \int_B [g+\Lcal(f)] \cdot h \dd \mathm \\
& = \int_B g \cdot h \dd \mathm + \int_A \Lcal(f) \cdot \mathbf{1}_B h \dd \mathm \\
& = \int_B g \cdot h \dd \mathm + \int_A f \cdot (\mathbf{1}_B h) \circ T \dd \mathm \\
& = \int_B g \cdot h \dd \mathm + \int_{A_1} f \cdot h \circ T \dd \mathm \\
& = \int_B g \cdot h \dd \mathm + \int_{B_1} f \cdot h \circ T \dd \mathm + \int_{C_1} f \cdot h \circ T \dd \mathm,
\end{align*}
which is the induction basis. Now, assume that Equation~\eqref{eq:InvarianceRecurrence} holds for some $n \geq 1$.
Then, since $g\mathbf{1}_{C_n}=0$:
\begin{align}
\int_{C_n} f \cdot h \circ T^n \dd \mathm 
& = \int_A \Lcal(f) \cdot \mathbf{1}_{C_n} \cdot h \circ T^n \dd \mathm \nonumber\\
& = \int_A f \cdot \mathbf{1}_{C_n} \circ T \cdot h \circ T^{n+1} \dd \mathm \nonumber\\
& = \int_{A_{n+1}} f \cdot h \circ T^{n+1} \dd \mathm \nonumber\\
& = \int_{B_{n+1}} f \cdot h \circ T^{n+1} \dd \mathm + \int_{C_{n+1}} f \cdot h \circ T^{n+1} \dd \mathm\, , \label{AAA}
\end{align}
which is the induction step. Hence, Equation~\eqref{eq:InvarianceRecurrence} holds for all $n \geq 1$.

\medskip
\textbf{\textsc{Convergence for $p = \infty$}}
\smallskip

Assume that $\mathm (B) < +\infty$. Note that $f = \mathbf{1}_A$, $g = \mathbf{0}$ is a solution of 
Equation~\eqref{eq:PoissonOriginal}. Taking $h=\mathbf{1}_B$ in Equation~\eqref{eq:InvarianceRecurrence} 
yields:
\begin{equation*}
\mathm(B)
= \sum_{k=1}^n \mathm (B_k) + \mathm(C_n),
\end{equation*}
so that $\mathm (C_n) = \mathm (B \cap \{\varphi > n\})$. Since $\varphi < + \infty$ almost everywhere on $B$, 
we get that $\lim_{n \to + \infty} \mathm (C_n) = 0$.

\smallskip

Let $f,g\in \Lbb^\infty(A,\mathm)$ with $g\equiv 0$ on $B^c$ and $(I-\Lcal)f=g$. Given $h \in \Lbb^1 (A, \mathm)$, write $\omega_h (\varepsilon) := \sup \{\int_E |h| \dd \mathm: \ \mathm (E) \leq \varepsilon\}$. 
Since $h$ is integrable, $\lim_0 \omega_h = 0$, and:
\begin{equation*}
\limsup_{n \to + \infty} \left| \int_{C_n} f \cdot h \circ T^n \dd \mathm \right| 
\leq \norm{f}{\Lbb^\infty (A, \mathm)} \limsup_{n \to + \infty} \omega_h (\mathm (C_n)) 
= 0,
\end{equation*}
so that, taking the limit in Equation~\eqref{eq:InvarianceRecurrence}, 
\begin{equation}
\label{eq:PoissonLimiteIntermediaire}
\int_B f \cdot h \dd \mathm 
= \int_B g \cdot h \dd \mathm + \lim_{n \to + \infty} \int_B \mathbf{1}_{\{\varphi \leq n\}}  f \cdot h \circ T_B \dd \mathm.
\end{equation}
Applying Equation~\eqref{eq:PoissonLimiteIntermediaire} to $g:=0$, $\tilde{f}:=\norm{f}{\Lbb^\infty (A, \mu)}$ and 
$\tilde{h} := |h|$, by monotone convergence, 
\begin{equation*}
 \lim_{n \to + \infty} \int_B \mathbf{1}_{\{\varphi \leq n\}}  \norm{f}{\Lbb^\infty (A, \mu)} \cdot |h| \circ T_B \dd \mathm 
 = \int_B  \norm{f}{\Lbb^\infty (A, \mu)} \cdot |h| \circ T_B \dd \mathm.
\end{equation*}
The general case follows by the dominated convergence theorem:
\begin{equation}
\label{eq:PreservationMesureFinie}
\int_B f \cdot h \dd \mathm 
= \int_B g \cdot h \dd \mathm + \int_B f \cdot h \circ T_B \dd \mathm.
\end{equation}

This concludes the case $p = \infty$ and $0 < \mathm (B) < +\infty$. 
Now, assume that $m(B)=+\infty$. Let $h \in \Lbb^1 (A, \mathm)$ be non-negative. 
Equation~\eqref{eq:InvarianceRecurrence} applied with $\mathbf{1}_A$ and $\mathbf 0$ implies, for all $n \geq 1$:
\begin{equation*}
\int_B h \dd \mathm 
\geq \sum_{k=1}^n \int_{B_k} h \circ T^k \dd \mathm,
\end{equation*}
whence, by taking the limit, we only get $\int_B h \dd \mathm \geq \int_B h \circ T_B \dd \mathm$. 
We want to show that mass can't escape from $B$ when working backwards\footnote{Following~\cite{Revuz:1975}, 
it may be more intuitive to think about the Markov chain dual to $T$, that is, the Markov chain with 
transition kernel $\Lcal$ on $A$. Then one can talk about forward orbits instead of backward orbits, 
and the goal is to prove that this dual Markov chain comes back almost surely to $B$ -- or, in other words, is recurrent.}. 
To prove this, starting from $\widetilde{B} \subset B$ with finite measure, 
we go backwards in time until we reach $B$, then work backwards again until we reach $\widetilde{B}$. 
Since we must eventually end up in $\widetilde{B}$, we must also eventually reach $B$.

\smallskip

Let $\varepsilon > 0$, and $\widetilde{B} \subset B$ be such that $\int_{B \setminus \widetilde{B}} h \dd \mathm \leq \varepsilon$ and 
$0 < \mathm (\widetilde{B}) < +\infty$. Define $\widetilde{B}_n$ as in the beginning of this proof, 
replacing $B$ by $\widetilde{B}$. Fix $n \geq 1$. For $1 \leq k \leq \ell \leq n$, let $B_{k, \ell} := T^{-(\ell-k)} (B_k) \cap \widetilde{B}_\ell$. 
Then:
\begin{align*}
\int_{B_k} h \circ T^k \dd \mathm 
& = \int_{B_{k, k}} h \circ T^k \dd \mathm + \int_{B_k \setminus B_{k, k}} h \circ T^k \dd \mathm \\
& = \int_{B_{k, k}} h \circ T^k \dd \mathm + \int_{T^{-1} B_k \setminus T^{-1} B_{k, k}} h \circ T^{k+1} \dd \mathm \\ 
& = \int_{B_{k, k}} h \circ T^k \dd \mathm + \int_{B_{k, k+1}} h \circ T^{k+1} \dd \mathm + \int_{T^{-1} B_k \setminus (T^{-1} B_{k, k} \cup B_{k,k+1})} h \circ T^{k+1} \dd \mathm \\ 
& = \ldots \\
& = \sum_{\ell = k}^n \int_{B_{k, \ell}} h \circ T^\ell \dd \mathm + \int_{T^{-(n-k)} B_k \setminus \bigcup_{\ell=k}^n T^{-(n-\ell)} B_{k, \ell}} h \circ T^n \dd \mathm \\
& \geq \sum_{\ell = k}^n \int_{B_{k, \ell}} h \circ T^\ell \dd \mathm.
\end{align*}
Since $\widetilde{B} \subset B$, for each $\ell$, we have $\widetilde{B}_\ell = \bigsqcup_{k=1}^\ell B_{k, \ell}$. 
In other words, the set of orbits starting from $\widetilde{B}$ at time $0$ and ending in 
$B$ at time $n$ can be partitioned depending on the last time $n-k<n$ 
at which they hit $B$. Hence,
\begin{equation*}
\sum_{k=1}^n \int_{B_k} h \circ T^k \dd \mathm
\geq \sum_{\ell=1}^n \int_{\widetilde{B}_\ell} h \circ T^\ell \dd \mathm.
\end{equation*}
By the monotonous convergence theorem, the two terms in the above
inequality converge respectively to $\int_{B} h \circ T_{B} \dd \mathm$
and to $\int_{\widetilde{B}} h \circ T_{\widetilde{B}} \dd \mathm = \int_{\widetilde{B}} h \dd \mathm$
(thanks to the finite measure case, Equation~\eqref{eq:PreservationMesureFinie}).
Hence, due to \eqref{eq:PreservationMesureFinie} (applied with $f=\mathbf{1}$ and $g=\mathbf 0$):
\begin{equation*}
\int_B h \dd \mathm 
\geq \int_B h \circ T_B \dd \mathm 
\geq \int_{\widetilde B} h \dd \mathm 
\geq \int_B h \dd \mathm -\varepsilon;
\end{equation*}
as this is true for all $\varepsilon > 0$, we get $\int_B h \dd \mathm = \int_B h \circ T_B \dd \mathm$. 
As a consequence, for all $h \in \Lbb^1 (A, \mathm)$, 
\begin{align*}
\limsup_{n \to + \infty} \left| \int_{C_n} f \cdot h \circ T^n \dd \mathm \right|
& \leq \norm{f}{\Lbb^\infty (A, \mathm)} \limsup_{n \to + \infty} \int_{C_n} |h| \circ T^n \dd \mathm \\
& = \norm{f}{\Lbb^\infty (A, \mathm)} \left[  \int_B |h| \dd \mathm - \int_B |h| \circ T_B \dd \mathm \right] 
= 0,
\end{align*}
which, due to \eqref{eq:PreservationMesureFinie}, yields:
\begin{align*}
\int_B f \cdot h \dd \mathm 
& = \int_B g \cdot h \dd \mathm + \int_B f \cdot h \circ T_B \dd \mathm \\
& = \int_B \left[g + \Lcal_B (f_{|B})\right] \cdot h \dd \mathm.
\end{align*}
As this is true for all $h \in \Lbb^1 (A, \mathm)$, we finally get Equation~\eqref{eq:PoissonInduit}.

\medskip
\textbf{\textsc{Convergence for $p < \infty$}}
\smallskip

Assume that $f \in \Lbb^p (A, \mathm)$ and let $h \in \Lbb^q (A, \mathm)$. Then:
\begin{equation*}
\left| \int_{C_n} f \cdot h \circ T^n \dd \mathm \right| 
\leq \norm{f \mathbf{1}_{C_n}}{\Lbb^p (A, \mathm)} \norm{h}{\Lbb^q (A, \mathm)}.
\end{equation*}
By definition, $\norm{f \mathbf{1}_{C_n}}{\Lbb^p (A, \mathm)}^p = \int_{C_n} |f|^p \dd \mathm$. 
Since $|f|^p$ is integrable and $\lim_{n\rightarrow +\infty} \mathm(C_n)=0$, by the Lebesgue dominated
convergence theorem, $\lim_{n \to + \infty} \norm{f \mathbf{1}_{C_n}}{\Lbb^p (A, \mathm)} = 0$,
and so that the above quantity converges to $0$ as $n$ goes to infinity. Moreover 
\begin{equation*}
 \sum_{k=1}^n \int_{B_k} f \cdot h \circ T^k \dd \mathm
 = \int_B \mathbf{1}_{\{\varphi \leq n\}} f \cdot h \circ T_B \dd \mathm\, .
\end{equation*}
Due to the case $p=\infty$ with $g = \mathbf{0}$ and $f = \mathbf{1}_A$, we know that $T_B$ preserves $\mathm_{|B}$ 
(see also Corollary~\ref{cor:InvarianceMesure}). Hence, 
$h\circ T_B$ is in $\Lbb^q (M,\mathm_{|B})$, with $\norm{h \circ T_B}{\Lbb^q(B,\mathm_{|B})} = \norm{h}{\Lbb^q(B,\mathm_{|B})}$. 
From there, we know that $f\cdot (h\circ T_B)$ is integrable with respect to $\mathm_{|B}$. 
Therefore, by the dominated convergence theorem, the above sum converges to 
$\int_B f\cdot h\circ T_B\dd\mathm$ and so, due to Equation~\eqref{eq:InvarianceRecurrence},
\begin{equation*}
 \int_B f\cdot h\dd\mathm
 = \int_B g\cdot h\dd\mathm + \int_B f\cdot h\circ T_B\dd\mathm\, .
\end{equation*}
This equation can be rewritten as:
\begin{equation*}
 \int_B (I-\Lcal_B)f_{|B}\cdot h\dd\mathm
 = \int_B g\cdot h\dd\mathm \quad \forall h\in \Lbb^q(m_B) \,,
\end{equation*}
and thus $(I-\Lcal_B) f_{|B} = g_{|B}$ in $\Lbb^p(m_B)$.
\end{proof}

As a corollary to the $p = \infty$ case, we do get the classical result that 
$(B, \mathm_{|B}, T_B)$ is measure-preserving:

\begin{corollary}
\label{cor:InvarianceMesure}

Let $(A, \mathm, T)$ be measure-preserving, with $\mathm$ a recurrent $\sigma$-finite measure. 
Let $B \subset A$ be such that $\mathm(B) > 0$. Then $(B, \mathm_{|B}, T_B)$ is measure-preserving.
\end{corollary}

\begin{proof}

Take again $g = \mathbf{0}$ and $f = \mathbf{1}_A$ in Proposition~\ref{prop:InvarianceInduction}. These 
two functions satisfy Equation~\eqref{eq:PoissonOriginal}, so:
\begin{equation*}
\mathbf{1}_B = \Lcal_B (\mathbf{1}_B),
\end{equation*}
which means that $\mathm_{|B}$ is $T_B$-invariant.
\end{proof}

\begin{remark}[Coboundaries]
\label{rmk:Cobords}

Let $(A, \mathm, T)$ be measure-preserving, with $\mathm$ a recurrent $\sigma$-finite measure. 
Let $B \subset A$ be such that $\mathm(B) > 0$. Assume that $T$ is invertible. Then the 
Koopman operator for $T$ is the transfer operator for the transformation $T^{-1}$. 
Proposition~\ref{prop:InvarianceInduction} implies that for $f$, $g \in \Lbb^p (A, \mathm)$ 
with $g$ supported on $B$, if $f-f\circ T = g$, then $f_{|B}-f_{|B} \circ T_B = g_{|B}$.

\smallskip

This proposition is actually true without any assumption on the integrability of $f$ and $g$, 
as well as for non-invertible $T$ (as it can be seen by a direct proof).

\smallskip

One of the strengths of Proposition~\ref{prop:InvarianceInduction} is that, for hyperbolic 
transformations, regular functions $g$ such that $f-f\circ T = g$ for some $f$ may be rare. 
For instance, as a consequence of Liv\v{s}ic's theorem~\cite{Livsic:1971}, a H\"older observable on an Anosov system is a coboundary 
if and only if its average on all periodic orbits vanishes, which brings countably many obstructions. 
On the other hand, for nice ergodic, non-invertible, hyperbolic systems, the transfer operator 
$\Lcal$ acts on spaces of regular functions (for instance H\"older functions). Then $1$ is a 
simple eingenvalue isolated in the spectrum of $\Lcal$, so the only obstruction to the existence of a function $f$ 
such that $f-\Lcal (f) = g$ is that the global average of $g$ must vanish.
\end{remark}

\begin{remark}[Recurrence and first return time]
\label{rmk:Recurrence}

Let $(A, \mathm, T)$ be measure-preserving, and let $B \subset A$ with $\mathm (B) > 0$. 
The hypothesis of recurrence in Proposition~\ref{prop:InvarianceInduction} can \textit{a priori} 
be weakened. A second look at its proof shows that the conclusion of the proposition holds:
\begin{itemize}
 \item if $\mathm (B) < + \infty$ and $\varphi_{B|B}<+\infty$ almost everywhere;
 \item if $\mathm (B) = + \infty$ and there exists an exhaustive sequence $(B_n)$ 
 of subsets of $B$ with $\varphi_{B_n|B_n}<+\infty$ almost everywhere,
\end{itemize}
where a sequence $(B_n)$ of subsets of $B$ is said to be \textit{exhaustive} if 
it is nondecreasing for the inclusion, all the $B_n$ have finite measure, and 
$\bigcup_{n \geq 0} B_n = B$ almost everywhere.

\smallskip

As a consequence of Corollary~\ref{cor:InvarianceMesure} and Poincar\'e's recurrence lemma, 
if $A$ itself satisfies either of the hypotheses above, then $(A, \mathm, T)$ is recurrent. 
This alternative set of hypotheses is thus not more general than recurrence, 
but can be convenient, for instance when working with Markov chains in Subsection~\ref{subsec:DynamiqueToMarkov}.
\end{remark}

\subsection{Existence and uniqueness of solutions for Poisson equation}
\label{subsec:EquationPoisson}

Given a function $g$, Equations~\eqref{eq:PoissonOriginal} and~\eqref{eq:PoissonInduit} are Poisson equations in $f$. 
In this subsection, we investigate the existence and uniqueness of its solution, as well as some elementary properties 
such as a maximum principle. We begin with uniqueness.

\begin{lemma}
\label{lem:Liouville}

Let $(A, \mathm, T)$ be measure-preserving, with $\mathm$ a recurrent $\sigma$-finite measure. 
Let $p \in [1, \infty]$, and $f_1$, $f_2$, $g \in \Lbb^p (A, \mathm)$ such that:
\begin{equation*}
(I-\Lcal) f_1
= (I-\Lcal) f_2 
= g.
\end{equation*}
Then $f_1-f_2$ is $T$-invariant almost everywhere. In particular, if $(A, \mathm, T)$ is also ergodic, 
then $f_1-f_2$ is constant almost everywhere. 
\end{lemma}

\begin{proof}

Working with $f_1-f_2$, it is enough to prove that any solution $f\in\Lbb^p$ of the 
equation $f = \Lcal (f)$ is $T$-invariant. 

\smallskip

Assume that $\mathm(A) < +\infty$. By Hurewicz' ergodic theorem~\cite{Hurewicz:1944}, there exists a $T$-invariant function $h$ such that:
\begin{equation*}
f
= \lim_{n \to + \infty} \frac{1}{n} \sum_{k=0}^{n-1} \Lcal^k (f) 
= h.
\end{equation*}
Hence, $f$ is $T$-invariant.

\smallskip

Now, assume that $\mathm (A) = + \infty$. Let $B \subset A$ be such that $0 < \mathm (B) < + \infty$. 
By Proposition~\ref{prop:InvarianceInduction}, $f_{|B} = \Lcal_B (f_{|B})$; by the finite measure case, 
$f_{|B}$ is $T_B$-invariant. As this is true for all $B \subset A$ with finite measure, $f$ is 
$T$-invariant.
\end{proof}

Now, let us study the existence of solutions to Equation~\eqref{eq:PoissonOriginal}. 
If $(A, \mathm, T)$ is ergodic and sufficiently hyperbolic and $g$ is smooth with average $0$, 
such solutions can be found. For instance, assume that there exists a Banach space 
$\mathbf{1}_A \in \Bcal \subset \Lbb^1 (A, \mathm)$ such that $\Lcal$ acts quasi-compactly on $\Bcal$. 
This is the case, for instance, if $(A, \mathm, T)$ is a subshift of finite type and $\Bcal$ is the 
space of Lipschitz functions~\cite{Bowen:1975}, or $(A, \mathm, T)$ is a piecewise expanding map 
of the interval and $\Bcal$ is the space of functions with bounded variation~\cite{LasotaYorke:1973}. 
Then, by ergodicity, $1$ is a single eigenvalue of $\Lcal$ corresponding to constant functions, 
and $\Lcal$ preserves $\Bcal_0 := \{f \in \Bcal : \int_A f \dd \mathm = 0\}$. 
Then $(I-\Lcal)$ is invertible on $\Bcal_0$. Hence, for all $g \in \Bcal_0$, there exists 
a solution $f \in \Bcal_0$ of Equation~\eqref{eq:PoissonOriginal}.

\smallskip

However, a general theme when using induction for dynamical system is that even if the initial system $(A, \mathm, T)$ 
is not hyperbolic, a well-chosen induced system $(B, \mathm, T_B)$ might be. That is, we may not find 
such a nice Banach space for $(A, \mathm, T)$, but still have one for $(B, \mathm, T_B)$. Hence, a method to prove the 
existence of solutions to Equation~\eqref{eq:PoissonOriginal} is to find a solution of Equation~\eqref{eq:PoissonInduit}, 
and to extend it to a solution of Equation~\eqref{eq:PoissonOriginal}. The following lemma states that this is possible, 
with a control on the $\Lbb^\infty$ norm of the extension (also known as a maximum principle).

\begin{lemma}
\label{lem:PoissonExtension}

Let $(A, \mathm, T)$ be measure-preserving, with $\mathm$ a $\sigma$-finite measure, ergodic and recurrent. 
Let $B \subset A$ be such that $0<\mathm(B) \leq +\infty$.

\smallskip

Let $f_B \in \Lbb^\infty (B, \mathm)$, $g \in \Lbb^\infty (A, \mathm)$ be such that $g \equiv 0$ on $B^c$ and:
\begin{equation*}
(I-\Lcal_B) f_B = g_{|B}.
\end{equation*}
Then there is a unique function $f \in \Lbb^\infty (A, \mathm)$ such that $f_{|B} = f_B$ and:
\begin{equation*}
(I-\Lcal) f = g.
\end{equation*}
In addition, $\norm{f}{\Lbb^\infty (A, \mathm)} = \norm{f_B}{\Lbb^\infty (B, \mathm)}$.
\end{lemma}

\begin{proof}

Let $g$, $f_B$ be as in the lemma. Without loss of generality, assume that $f \geq 0$.
For all $n \geq 1$, let $A_n := \{\varphi=n\}$, 
$B_n := A_n \cap B$ and $C_n := A_n \cap B^c$. Recurrence and ergodicity ensure that 
$\{B, C_n: \ n \geq 1\}$ is a $\mu$-essential partition of $A$. Define $f$ on $B$ by $f \mathbf{1}_B := f_B$, 
and on $C_n$ by:
\begin{equation*}
f \mathbf{1}_{C_n} 
:= \sum_{k=n+1}^{+\infty} \Lcal^{k-n} (f_B \mathbf{1}_{B_k}).
\end{equation*}

\smallskip

Let $n \geq 1$. Extend $f_B$ to a function $f_{B \cup C_n}$ defined on $B \cup C_n$ by 
$f_{B \cup C_n} \mathbf{1}_{C_n} \equiv 0$. Then $f_{|C_n} = \Lcal_{B\cup C_n} (f_{B \cup C_n})_{|C_n}$. 
By Proposition~\ref{prop:InvarianceInduction}, $\Lcal_{B\cup C_n}$ is a weak contraction when acting on $\Lbb^\infty (B \cup C_n, \mathm)$, so that 
$\norm{f_{|C_n}}{\Lbb^\infty (C_n, \mathm)} \leq \norm{f_{B \cup C_n}}{\Lbb^\infty (B \cup C_n, \mathm)} = \norm{f_B}{\Lbb^\infty (B, \mathm)}$, 
and:
\begin{equation*}
\norm{f}{\Lbb^\infty (A, \mathm)} 
= \sup_{D \in \{B, C_n: \ n \geq 1\}} \norm{f}{\Lbb^\infty (D, \mathm)} 
= \norm{f_B}{\Lbb^\infty (B, \mathm)}.
\end{equation*}

\smallskip

Now, we check that $f = \Lcal (f)+g$. Let $n \geq 1$. Since $T^{-1} (C_n) = A_{n+1}$, 
\begin{align*}
\Lcal (f) \mathbf{1}_{C_n} 
& = \Lcal (f \mathbf{1}_{A_{n+1}}) 
= \Lcal (f \mathbf{1}_{B_{n+1}}) + \Lcal \left( \sum_{k=n+2}^{+\infty} \Lcal^{k-n-1} (f_B \mathbf{1}_{B_k}) \right) \\
& = \sum_{k=n+1}^{+\infty} \Lcal^{k-n} (f_B \mathbf{1}_{B_k}) 
= f \mathbf{1}_{C_n}.
\end{align*}
Since this is true for all $n$, we have $f = \Lcal (f) = \Lcal (f)+g$ on $B^c$. Since 
$T^{-1} (B) = A_1$, in the same way,
\begin{equation*}
\Lcal (f) \mathbf{1}_B 
= \sum_{k=1}^{+\infty} \Lcal^k (f_B \mathbf{1}_{B_k}) 
= \Lcal_B (f_B) 
= f_B-g
= f \mathbf{1}_B-g.
\end{equation*}
Hence, $f = \Lcal(f)+g$ almost everywhere on $A$.

\smallskip

The uniqueness of $f$ comes from Lemma~\ref{lem:Liouville}.
\end{proof}

Finally, let us state Propositions~\ref{prop:InvarianceInduction} and Lemma~\ref{lem:Liouville} in another way. 
Let $(A, \mathm, T)$ be measure-preserving, with $\mathm$ a $\sigma$-finite measure, ergodic and recurrent. 
Let $B \subset A$ be such that $0 < \mathm(B) < +\infty$. Let $\Bcal (B, \mathm)$ be a subspace of $\Lbb^\infty (B,\mathm)$, 
and $\Bcal_0 (B, \mathm) := \{g \in \Bcal (B, \mathm) : \ \int_B g \dd \mathm = 0\}$. Let 
$\Gamma : \Bcal_0 (B, \mathm) \to \Lbb^\infty (A, \mathm)$ be a right inverse of $(I-\Lcal)$ on $V$:
\begin{equation*}
 (I-\Lcal) \Gamma g 
 = g \ \ \forall g \in \Bcal_0 (B, \mathm).
\end{equation*}
Let $\Gamma_B : \Bcal_0 (B, \mathm) \to \Lbb^\infty (B, \mathm)$ be a right inverse of $(I-\Lcal_B)$ on $\Bcal_0 (B, \mathm)$.
Then $\Gamma$ may not coincide with $\Gamma_B$ on $B$, because these inverses may differ by a constant. 
However, this ambiguity vanishes if we integrate against an observable $h$ with average $0$ 
and supported on $B$:
\begin{equation}
\label{eq:InvarianceInductionIntegrale}
 \int_A \Gamma (g) \cdot h \dd \mathm 
 = \int_B \Gamma_B (g_{|B}) \cdot h_{|B} \dd \mathm \ \ \forall g \in \Bcal_0 (B, \mathm), \ \forall h \in \Lbb_0^1 (B,\mathm).
\end{equation}

\subsection{Operators with weights}
\label{subsec:CasPondere}

One may want more flexibility in the choice of operators they work with, and in particular use 
general weighted operators instead of the transfer operator $\Lcal$. In order to keep the discussion 
elementary, we put very stringent conditions on these weights.

\smallskip

As in Subsection~\ref{subsec:CasGeneral}, let $(A, \mathm, T)$ be a recurrent measure-preserving dynamical system. 
Let $B \subset A$ be measurable, with $0 < \mathm (B) < + \infty$. Given a measurable function $\phi : A \to \C$ 
such that $\Esup \Re (\phi) < + \infty$, let:
\begin{equation*}
\Lcal_\phi (f) 
:= \Lcal (e^\phi f) \quad \forall f \in \Lbb^p (A, \mathm).
\end{equation*}

We also write $S_{\varphi_B} \phi (x) := S_{\varphi_B (x)} \phi (x)$.

\begin{lemma}
\label{lem:InvarianceInductionPonderee}

Let $(A, \mathm, T)$ be measure-preserving, with $\mathm$ a recurrent $\sigma$-finite measure. 
Let $B \subset A$ be such that $\mathm(B) > 0$. Let $\phi : A \to \C$ be measurable, 
with:
\begin{equation*}
 \Esup_{n \geq 0} S_n \Re (\phi)
 < +\infty.
\end{equation*}
Let $p \in [1, \infty]$ and $f$, $g \in \Lbb^p (A, \mathm)$ be such that $g \equiv 0$ on $B^c$ and:
\begin{equation*}
f = \Lcal_\phi (f)+g.
\end{equation*}

Then:
\begin{equation*}
f_{|B} = \Lcal_{B,(S_{\varphi_B} \phi)} (f_{|B})+g_{|B}\, , \text{ where } \Lcal_{B,\psi}:=\Lcal_B(e^\psi\cdot)\, .
\end{equation*}
\end{lemma}

\begin{proof}

The proof of this lemma mirrors the proof of Proposition~\ref{prop:InvarianceInduction}.
Let $f$, $g$ and $\phi$ be as in the lemma. Let $q$ be the conjugate of $p$, and $h \in \Lbb^q (A, \mathm)$. 
Equation~\eqref{eq:InvarianceRecurrence} becomes, for all $n \geq 1$:
\begin{equation}
\label{eq:InvarianceRecurrencePonderee}
\int_B f \cdot h \dd \mathm 
= \int_B g \cdot h \dd \mathm + \sum_{k=1}^n \int_{B_k} f e^{S_k \phi} \cdot h \circ T^k \dd \mathm  + \int_{C_n} f e^{S_n \phi} \cdot h \circ T^n \dd \mathm.
\end{equation}
Note that:
\begin{equation*}
 \left| \int_{C_n} f e^{S_n \phi} \cdot h \circ T^n \dd \mathm \right|
 \leq e^{\Esup_{n \geq 0} S_n \Re (\phi)} \int_{C_n} |f| \cdot |h| \circ T^n \dd \mathm 
 \to_{n \to + \infty} 0,
\end{equation*}
where the limit as $n$ goes to infinity follows from the arguments in the proof of Proposition~\ref{prop:InvarianceInduction}. 
Finally, 
\begin{equation*}
 \sum_{n=1}^{+ \infty} \int_{B_n} f e^{S_n \phi} \cdot h \circ T^n \dd \mathm 
 = \int_B \sum_{n=1}^{+ \infty} \Lcal^n (e^{S_n \phi} \mathbf{1}_{B_n} f) \cdot h \dd \mathm 
 = \int_B  \Lcal_{B, (S_{\varphi_B} \phi)} (f) \cdot h \dd \mathm,
\end{equation*}
where the series converges in $\Lbb^p (B, \mathm)$.
\end{proof}

Lemma~\ref{lem:InvarianceInductionPonderee} is especially interesting when applied to the so-called 
\emph{geometrical weights}; for instance, if $T$ is an Axiom A diffeomorphism with strong unstable 
direction $E^u$ and $s \in \C$ with $\Re (s) \geq 0$, one may take $\phi (x) = - s \ln |\det ((D_x T)_{|E^u})|$. Then 
the tangent space of $B$ also admits a splitting into stable and unstable directions for $T_B$, 
and by the chain rule, for all recurrent $x \in B$:
\begin{equation*}
 - s \ln |\det ((D_x T_B)_{|E^u})| 
 = (S_{\varphi_B} \phi) (x).
\end{equation*}
The potential $S_{\varphi_B} \phi$ admits the same geometric interpretation as $\phi$, 
but for $T_B$ instead of $T$.

\subsection{Functions without localisation}
\label{subsec:SansLocalisation}

Proposition~\ref{prop:InvarianceInduction} uses crucially the hypothesis that $(I-\Lcal) f = g$ 
is supported on the set $B$ on which we induce. However, in practice, one may want to induce multiple times, 
or on small sets, in which case this hypothesis may prove inconvenient. We now discuss what happens 
when one omits this localisation hypothesis, giving a variant of~\cite[Exercise~2.16]{Revuz:1975}. 
In order to simplify the discussion, we restrict ourselves to inducing sets $B$ with finite measure.

\begin{proposition}
\label{prop:InvarianceInductionNonLocalisee}

Let $(A, \mathm, T)$ be measure-preserving, with $\mathm$ a recurrent $\sigma$-finite measure. 
Let $B \subset A$ be such that $0 < \mu (B) < +\infty$. For $n \geq 1$, let $C_n := \{\varphi_B = n\} \cap B^c$.

\smallskip

Let $p \in [1, \infty]$ and $f$, $g \in \Lbb^p (A, \mathm)$ be such that:
\begin{equation*}
 (I-\Lcal) f 
 = g,
\end{equation*}
and:
\begin{equation}
\label{eq:HypotheseSupplementaireG}
 \sum_{k \geq 1} \norm{\Lcal^k(\mathbf{1}_{C_k} g)}{\Lbb^p (B, \mathm)}
 < +\infty.
\end{equation}

Then:
\begin{equation*}
 (I-\Lcal_B) f_{|B} 
 = g_{|B} + \sum_{n \geq 1} \Lcal^k(\mathbf{1}_{C_k} g).
\end{equation*}
\end{proposition}

\begin{proof}

Let $f$, $g$ be as in the proposition. For all $n \geq 1$, let $A_n := \{\varphi = n\}$ 
and $B_n := A_n \cap B$. Denote by $q$ the conjugate of $p$. Without the hypothesis that $g$ 
vanishes on $B^c$, for all $h \in \Lbb^q (A, \mathm)$, for all $n \geq 1$, Equation~\eqref{AAA} becomes
\begin{equation*}
 \int_{C_n} f \cdot h \circ T^n \dd \mathm 
 = \int_{C_n} g \cdot h \circ T^n \dd \mathm +  \int_{B_{n+1}} f \cdot h \circ T^{n+1} \dd \mathm + \int_{C_{n+1}} f \cdot h \circ T^{n+1} \dd \mathm,
\end{equation*}
whence Equation~\eqref{eq:InvarianceRecurrence} becomes:
\begin{equation}
\label{eq:InvarianceRecurrenceNonLocalisee}
\int_B f \cdot h \dd \mathm 
= \int_B g \cdot h \dd \mathm + \sum_{k=1}^n \int_{B_k} f \cdot h \circ T^k \dd \mathm + 
\sum_{k=1}^{n-1} \int_{C_k} g \cdot h \circ T^k \dd \mathm + \int_{C_n} f \cdot h \circ T^n \dd \mathm.
\end{equation}
The convergence of the right hand-side as $n$ goes to infinity works out as in the proof of
Proposition~\ref{prop:InvarianceInduction} for the terms involving $f$, and uses the 
hypothesis on $g$ for the remaining term. We get:
\begin{equation*}
\int_B f \cdot h \dd \mathm 
= \int_B \left[ g + \Lcal_B (f_{|B}) + \sum_{n \geq 1} \Lcal^n (\mathbf{1}_{C_n} g) \right] \cdot h \dd \mathm.
\end{equation*}
This equation holds for all $h \in \Lbb^q (A, \mathm)$, from which Proposition~\ref{prop:InvarianceInductionNonLocalisee} 
follows.
\end{proof}

For $p = 1$, the additional hypothesis~\eqref{eq:HypotheseSupplementaireG} on $g$ comes for free. Indeed,
\begin{equation*}
 \sum_{n \geq 1} \norm{\Lcal^n(\mathbf{1}_{C_n} g)}{\Lbb^1 (B, \mathm)} 
 = \sum_{n \geq 1} \int_{C_n} |g| \dd \mathm 
 = \int_{B^c} |g| \dd \mathm 
 \leq \norm{g}{\Lbb^1 (A, \mathm)}.
\end{equation*}
In particular, if $\mathm (A) < +\infty$, then one can work in $\Lbb^1 (A, \mathm)$ and 
immediately get that:
\begin{equation*}
 (I-\Lcal_B) f_{|B} 
 = g_{|B} + \sum_{n \geq 1} \Lcal^k(\mathbf{1}_{C_k} g),
\end{equation*}
where the series converges in $\Lbb^1 (B, \mathm)$. If $f$ and $g$ belong to 
$\Lbb^1 (A, \mathm) \cap \Lbb^p (A, \mathm)$, then the series is also in 
$\Lbb^p (B, \mathm)$.

\smallskip

However, proving directly that the series $\sum_{n \geq 1} \Lcal^k(\mathbf{1}_{C_k} g)$ 
converges in $\Lbb^p (B, \mathm)$ for some $p > 1$ is more delicate. 
In the finite measure case, given $r < p$, Riesz-Thorin's inequality yields:
\begin{equation*}
 \sum_{n \geq 1} \norm{\Lcal^n(\mathbf{1}_{C_n} g)}{\Lbb^r (B, \mathm)} 
 \leq \norm{g}{\Lbb^p (A, \mathm)} \sum_{n \geq 1} 
\mathm (\varphi_B = n)^{\frac{1}{r} - \frac{1}{p}}\, ;
\end{equation*}
if the tails of $\varphi_B$ decay fast enough (which implies $\mathm (A) < +\infty$), 
this upper bound ensures that the series converges in $\Lbb^r (B, \mathm)$. In particular, 
if the tails of $\varphi_B$ decay exponentially fast, then the loss of integrability is arbitrarily small.

\smallskip

Depending on the system, one may improve these rough estimates. For instance, if $(A, \mathm, T)$ 
is a Gibbs-Markov map, $B$ is measurable with respect to the Markov partition and $g$ is integrable and locally 
Lipschitz with integrable Lipschitz seminorm, one may use a version~\cite[Proposition~4.6.2]{Aaronson:1997} together 
with~\cite[Lemme~1.1.13]{Gouezel:2008e} to get that the series in Equation~\eqref{eq:HypotheseSupplementaireG} 
converges with $p=\infty$.

\section{Relationship with recurrent Markov chains}
\label{sec:DynamiqueMarkov}

The tools developped in Section~\ref{sec:InvarianceInduction} are analogous to those developped 
for Markov chains, in particular in~\cite[Chapter~8.2]{Revuz:1975}. We aim to make this analogy explicit, 
in both directions:
\begin{itemize}
 \item by constructing adequate subshifts, a Markov chain can be encoded into a dynamical system. 
 We show what the consequences of Proposition~\ref{prop:InvarianceInduction} on Markov chains are.
 \item conversely, any dynamical system can be seen as a Markov chain 
 with transition kernel $P_x = \delta_{T(x)}$. We show how 
 to recover the main results of Section~\ref{sec:InvarianceInduction} from known results on Markov chains. 
\end{itemize}
As we shall see, these two points of views are mostly equivalent. Using Markov chains offers slightly 
more flexibility, but requires more background. First, let us state a few relevant definitions about Markov 
chains on general state spaces.

\subsection{Definitions and preliminary results}
\label{subsec:MarkovDefinitions}

Let $\Omega$ be a Polish space. A \emph{Markov transition kernel} on $\Omega$ is a 
map $P : \Omega \to \Pcal (\Omega)$ such that $x \mapsto P_x (B)$ 
is measurable for each measurable subset $B \subset \Omega$.
By Kolmogorov's extension theorem, there exists a Markov chain 
$(X_n)_{n \geq 0}$ with transition kernel $P$~\cite[Theorem~2.8]{Revuz:1975}.
In this section, we shall assume that $(X_n)_{n \geq 0}$ admits 
a $\sigma$-finite stationary measure $\mu$.

\smallskip

A Markov chain $(X_n)_{n \geq 0}$ on $\Omega$ with $\sigma$-finite 
stationary measure $\mu$ is said to be \emph{recurrent} if, for any measurable 
$\Psi \subset \Omega$ with $\mu (\Psi) > 0$, almost everywhere for $\mu_{|B}$, 
almost surely there exists some $n \geq 1$ such that $X_n \in \Psi$.

\smallskip

Given a recurrent Markov chain and $\Psi \subset \Omega$ with $\mu (\Psi) > 0$, 
the return time to $\Psi$ is defined by $T_\Psi := \inf \{n > 0: X_n \in \Psi\}$. 
The recurrence implies that $T_\Psi < + \infty$ almost everywhere. Then the sequence of 
hitting times $(0=: T_{\Psi, 0}, T_\Psi=: T_{\Psi, 1}, T_{\Psi, 2}, \ldots)$ 
is well-defined almost everywhere on $\Psi$. By the strong Markov property, if $X_0 \in \Psi$ then 
$(X_{T_{\Psi, n}})_{n \geq 0}$ is a Markov chain, which we shall call the \emph{Markov chain induced on} $\Psi$. 
We shall denote its transition kernel by $P_\Psi$.

\smallskip

A Markov transition kernel $P$ can be seen as an operator acting on the set of measurable bounded functions on $\Omega$, 
or, neglecting what happens on $\mu$-negligible sets, on $\Lbb^\infty (\Omega, \mu)$ by:
\begin{equation*}
 P (h) (\omega) 
 := \int_\Omega h (\omega') \dd P_\omega (\omega') 
 = \Ebb_\omega (f(X_1)). 
\end{equation*}
If $\mu$ is stationary, then $P$ acts on $\Lbb^p (\Omega, \mu)$ for all $p \in [1, +\infty]$. 
Conversely, one can recover $P$ from its action on $\Lbb^p (\omega, \mu)$ up to a $\mu$-negligible set. 
In what follows, we identify a Markov transition kernel (modulo $\mu$) with its action as an operator.

\smallskip

Given a Markov transition kernel $P$ and a stationary measure $\mu$, its dual transition kernel 
is the operator $P^*$ defined by:
\begin{equation*}
 \int_\Omega P(f) \cdot g \dd \mu 
 = \int_\Omega f \cdot P^*(g) \dd \mu
 \quad \forall f \in \Lbb^p (\Omega, \mu), \ \forall g \in \Lbb^q (\Omega, \mu),
\end{equation*}
where $1/p+1/q=1$. 
The following lemma states that recurrence transfers to the dual Markov chain. Its proof is close 
both to that of~\cite[Proposition~3.10]{Revuz:1975}, and of Proposition~\ref{prop:InvarianceInduction}.

\begin{lemma}
\label{lem:RecurrenceDual}
 
 Let $P$ be a Markov transition kernel on a Polish space $\Omega$. Let 
 $\mu$ be a $\sigma$-finite stationary measure, and $P^*$ be the dual Markov transition kernel.
 
 \smallskip
 
 If $\mu$ is recurrent for $P$, then it is recurrent for $P^*$.
\end{lemma}

\begin{proof}
 
 Let $\Omega$, $P$, and $\mu$ be as in the hypotheses. Let $\Psi \subset \Omega$ 
 with $0 < \mu (\Psi) < +\infty$. Let $f \in \Lbb^1 (\Omega, \mu)$ be non-negative 
 and supported on $\Psi$. Then, almost everywhere on $\Psi$:
 \begin{equation*}
  P_\Psi (f_{|\Psi}) 
  = \sum_{n \geq 0} \mathbf{1}_\Psi (P \mathbf{1}_{\Psi^c})^n (P \mathbf{1}_\Psi) (f)\, ,
 \end{equation*}
with convention $P\mathbf{1}_E:f\mapsto P(\mathbf{1}_E\, f)$ for any $E\subset\Omega$.
 Let $P^*$ be the Markov transition kernel dual to $P$ with respect to $\mu$. 
 For every $g \in \Lbb^\infty (\Omega, \mu)$, non-negative and supported on $\Psi$,
 \begin{equation}
  \label{eq:InvarianceInductionMarkov}
  \int_\Psi P_\Psi (f) \cdot g \dd \mu 
  = \sum_{n \geq 0} \int_\Psi \mathbf{1}_\Psi (P \mathbf{1}_{\Psi^c})^n (P \mathbf{1}_\Psi) (f) \cdot g \dd \mu 
  = \int_\Psi f \cdot \sum_{n \geq 0} \mathbf{1}_\Psi (P^* \mathbf{1}_{\Psi^c})^n (P^* \mathbf{1}_\Psi) (g) \dd \mu.
 \end{equation}
 Note that $P_\Psi (\mathbf{1}_\Psi) (\omega) = \Pbb_\omega (T_\Psi < +\infty) = 1$ almost everywhere on $\Psi$.
 Taking $f = g = \mathbf{1}_{\Psi}$, which is possible since $\Psi$ has finite measure, we get:
 \begin{equation*}
  \mu(\Psi)
  = \int_\Psi \sum_{n \geq 0} \mathbf{1}_\Psi (P^* \mathbf{1}_{\Psi^c})^n (P^* \mathbf{1}_\Psi) (\mathbf{1}_\Psi) \dd \mu 
  = \int_\Psi \Pbb_\omega (T^* < +\infty) \dd \mu (\omega),
 \end{equation*}
 where $T^*$ is the first hitting time of $\Psi$ for the dual Markov chain.
 Hence, $\Pbb_\omega (T^* < +\infty) \equiv \mathbf{1}$ $\mu$-almost everywhere on $\Psi$.
\end{proof}

As Corollary~\ref{cor:InvarianceMesure} follows from Proposition~\ref{prop:InvarianceInduction}, 
it follows from Lemma~\ref{lem:RecurrenceDual} that $\mu_{|\Psi}$ is $P_\Psi$-invariant. Indeed, 
taking $g = \mathbf{1}_\Psi$, for any $f \in \Lbb^1 (\Psi, \mu)$,
\begin{equation*}
 \int_\Psi P_\Psi (f) \dd \mu 
 = \int_\Psi f \cdot P_\Psi^* (\mathbf{1}_\Psi) \dd \mu 
 = \int_\Psi f \dd \mu.
\end{equation*}
As a consequence of Equation~\eqref{eq:InvarianceInductionMarkov}, $(P_\Psi)^* = (P^*)_\Psi$, where 
the first dual is taking with respect to $\mu_{|\Psi}$ and the second with respect to $\mu$.

\subsection{From dynamical systems to Markov chains}
\label{subsec:DynamiqueToMarkov}

Given a transition kernel $P$ with stationary measure $\mu$, a Markov chain can be realized 
on the \textit{canonical space} $A := \Omega^{\otimes \N}$ with the filtration 
$(\Fcal_n)_{n \geq 0} = (\sigma (X_0, \ldots, X_n))_{n \geq 0}$, 
and $X_n ((\omega_k)_{k\ge 0}) := \omega_n$. 
Introducing the shift map:
\begin{equation*}
 T := \left\{ 
 \begin{array}{lll}
  A & \to & A \\
  (\omega_n)_{n \geq 0} & \mapsto & (\omega_{n+1})_{n \geq 0}
 \end{array}
 \right.,
\end{equation*}
we see that $X_n = X_0 \circ T^n$. In addition, this construction 
yields a $\sigma$-finite $T$-invariant measure $\mathm$ on $A$, whose one-dimensional 
marginals are all $\mu$. In this subsection, we are given a Markov chain on $\Omega$ 
with transition kernel $P$ and recurrent stationary measure $\mu$, and $(A, \mathm, T)$ 
shall denote its realization on its canonical space.

\smallskip

Let $(\Psi_n)_{n \geq 0}$ be an exhaustive sequence of subsets of $\Omega$ (recall Remark~\ref{rmk:Recurrence} 
for the definition of exhaustivity), and define $B_n := \Psi_n \times \Omega^{\N_+} \subset A$. Then $(B_n)_{n \geq 0}$ 
is an exhaustive sequence for $A$, and the recurrence of the Markov chain implies that $\varphi_{B_n|B_n} < +\infty$ 
$\mu$-almost everywhere on $B_n$. By Remark~\ref{rmk:Recurrence}, we get the following result of independent interest:

\begin{lemma}
\label{lemma:RecurrenceDuale}
 
 Let $P$ be a Markov transition kernel on a Polish space $\Omega$, with $\sigma$-finite stationary measure $\mu$. 
 Let $(A, \mathm, T)$ be its realisation on its canonical space.
 
 \smallskip
 
 The measure $\mu$ is recurrent for $P$ if and only if the measure $\mathm$ is recurrent for $T$. 
\end{lemma}

Note that, if $\mu$ is finite, recurrence follows more directly from Poincar\'e's lemma. 
Hence we will be able to apply Proposition~\ref{prop:InvarianceInduction} to the system $(A, \mathm, T)$. 
For now, let us understand better the action of the transfer operator $\Lcal$.

\smallskip

For any $p \in [0, +\infty]$, there is an isometric embedding $\iota: \Lbb^p (\Omega, \mu) \hookrightarrow \Lbb^p (A, \mathm)$ 
defined by:
\begin{equation*}
 \iota (f) (\omega_0, \omega_1, \ldots) 
 := f(\omega_0).
\end{equation*}
Denote by $\Bcal^p (A, \mathm)$ its image; when working with functions in $\Bcal^p (A, \mathm)$, we may abuse notations 
and denote $\iota (f)$ by $f$.

\smallskip

Let $f \in \Bcal^p (A, \mathm)$ and $h \in \Lbb^q (A, \mathm)$ with $1/p+1/q=1$. Then:
\begin{equation*}
 \int_A f \cdot h \circ T \dd \mathm 
 = \int_A f (\omega_0) \cdot h (\omega_1, \ldots) \dd \mathm (\omega_0, \ldots) 
 = \int_A \frac{\de \left( \int_\Omega f (\omega_0) P_{\omega_0} h (\omega_1, \ldots) \dd \mu (\omega_0) \right)}{\de \mu} \dd \mathm (\omega_1, \ldots).
\end{equation*}
Hence, for $f \in \Bcal^p (A, \mathm)$,
\begin{equation}
\label{eq:DescriptionOperateur}
 \Lcal (f) 
 = \frac{\de \left( \int_\Omega f (\omega_0) P_{\omega_0} \dd \mu(\omega_0) \right)}{\de \mu}.
\end{equation}
Note that, using the construction of the measure $\mathm$ -- that is, the Markov property -- this formula defines a function on $\Omega$; 
that is, $\Lcal(f) \in \Bcal^p (A, \mathm)$. Hence, $\Lcal$ preserves $\Bcal^p (A, \mathm)$. 
Since $T$ acts by isometry on $\Lbb^q (A, \mathm)$, the operator $\Lcal$ is a weak contraction on $\Bcal^p (A, \mathm)$.

\smallskip

In what follows, we fix $p \in [1, +\infty]$, a subset $\Psi \subset \Omega$ such that $\mu (\Psi) >0$, 
and functions $f$, $g \in \Bcal^p (A, \mathm)$ such that:
\begin{equation*}
 (I-\Lcal) f 
 = g.
\end{equation*}
By Proposition~\ref{prop:InvarianceInduction}, 
\begin{equation*}
 (I-\Lcal_B) f_{|B} 
 = g_{|B}.
\end{equation*}
The goal is then to understand $\Lcal_B$, which can be done by studying the induced system $(B, \mathm_{|B}, T_B)$. 
Notice that, since $f_{|B}$ and $g_{|B}$ are both in $\{h \in \Bcal^p (A, \mathm) : \ \Supp (h) \subset B \}$, 
the operator $\Lcal_B$ acts on $\{h \in \Bcal^p (A, \mathm) : \ \Supp (h) \subset B \}$ by the equation above.

\smallskip

Let us define the set $\Ecal(\Psi)$ of excursions from $\Psi$:
\begin{equation*}
 \Ecal(\Psi) 
 := \bigsqcup_{n \geq 0} \Psi \times (\Psi^c)^n.
\end{equation*}
Then any point in $B = \Psi \times \Omega^{\N_+}$ whose $T$-orbit comes back infinitely often to $B$ 
can be written uniquely as a concatenation of words in $\Ecal(\Psi)$. We get a map $\phi$:
\begin{equation*}
 \phi : B \to \Ecal(\Psi)^\N,
\end{equation*}
which by recurrence is well-defined $\mathm_{|B}$-almost everywhere and injective. 
Let $S'$ be the shift on $\Ecal(\Psi)^\N$. Then $\phi$ induces an isomorphism of 
measured dynamical systems between $(B, \mathm_{|B}, T_B)$ and $(\Ecal(\Psi)^\N, \pi_* \mathm_{|B}, S')$.

\smallskip

In addition, the map $(\omega_0, \ldots, \omega_{n+1}) \mapsto \omega_0$ from $\Ecal(\Psi)$ 
to $\Psi$ induces a factor map $\pi: (\Ecal(\Psi)^\N, S') \to (\Psi^\N, S)$, with $S$ the 
shift on $\Psi^\N$. To sum up, the following diagram commutes:
\begin{equation*}
\begin{tikzcd}
B \arrow[d, "T_B"] \arrow[r, "\phi"]  &  \Ecal(\Psi)^\N \arrow[d, "S'"] \arrow[r, "\pi"]  &  \Psi^\N \arrow[d, "S"]  \\
B                  \arrow[r, "\phi"]  &  \Ecal(\Psi)^\N                 \arrow[r, "\pi"]  &  \Psi^\N
\end{tikzcd}
\end{equation*}
which implies that $\mathm_\Psi := \pi_* \phi_* \mathm_{|B}$ is $S$-invariant.

\smallskip

By construction, $\pi \circ \phi ((X_n)_{n \geq 0}) = (X_{T_{\Psi, n}})_{n \geq 0}$ is a Markov chain 
on $\Psi$ with transition kernel $P_\Psi$. Moreover, $[\pi \circ \phi ((X_n)_{n \geq 0})]_0 = X_0$, 
so the first marginal of $\mathm_\Psi$ is $\mu_{|\Psi}$, and $\mathm_\Psi$ is the $S$-invariant 
measure on the canonical space $\Psi^\N$ associated with the induced Markov chain $(X_{T_{\Psi, n}})_{n \geq 0}$ 
and the initial measure $\mu_{|\Psi}$. In particular, $\mu_{|\Psi}$ is stationary for the Markov chain 
induced on $\Psi$. 

\smallskip

Let $\Lcal_\Psi$ be the transfer operator for the system $(\Psi^\N, \mathm_\Psi, S)$.
Functions in $\Bcal^p (A, \mathm)$ supported on $B$ quotient through $\pi$ to get 
functions in $\Bcal^p (\Psi^\N, \mathm_\Psi)$. This yields a bijective isometry:
\begin{equation*}
 \pi_* : \{h \in \Bcal^p (A, \mathm) : \ \Supp (h) \subset B \} \to \Bcal^p (\Psi^\N, \mathm_\Psi).
\end{equation*}
The transformation $\pi_*$ conjugates the transfer operator $\Lcal_B$ on $\{h \in \Bcal^p (A, \mathm) : \ \Supp (h) \subset B \}$ 
with the transfer operator $\Lcal_\Psi$ on $\Bcal^p (\Psi^\N, \mathm_\Psi)$. By Equation~\eqref{eq:DescriptionOperateur},
\begin{equation*}
 \Lcal_B (f_{|B}) 
 = \frac{\de \left( \int_\Psi f (\omega_0) P_{\Psi, \omega_0} \dd \mu(\omega_0) \right)}{\de \mu}.
\end{equation*}
This discussion yields:

\begin{corollary}
\label{cor:InvarianceMarkovDuale}
 
 Let $\Omega$ be a Polish space, and $P$ a Markov transition kernel on $\Omega$. 
 Let $\mu$ be a $\sigma$-finite stationary and recurrent measure.
 
 \smallskip
 
 Let $\Psi \subset \Omega$ be measurable, with $\mu (\Psi) > 0$. Let $p \in [1, +\infty]$ 
 and $f$, $g \in \Lbb^p (\Omega, \mu)$ with $\Supp (g) \subset \Psi$.
 Assume that:
 \begin{equation*}
  (I-\Lcal) f = g.
 \end{equation*}
 Then:
 \begin{equation*}
  (I-\Lcal_\Psi) f_{|\Psi} = g_{|\Psi}.
 \end{equation*} 
\end{corollary}

The operators $P$ and $\Lcal$ are in duality. Indeed,  
for $f \in \Lbb^p (\Omega, \mu)$ and $h \in \Lbb^q (\Omega, \mu)$ with $1/p+1/q=1$,
\begin{align*}
 \int_\Omega f \cdot P(h) \dd \mu 
 & = \int_\Omega f (\omega) \int_\Omega h (\omega') \dd P_\omega (\omega') \dd \mu (\omega) \\
 & = \int_\Omega h (\omega') \frac{\de \left( \int_\Omega f (\omega) \dd P_\omega \dd \mu (\omega) \right) }{\de \mu} \dd \mu(\omega') \\
 & = \int_\Omega \Lcal (f) \cdot h \dd \mu.
\end{align*}
By Lemma~\ref{lemma:RecurrenceDuale}, the measure $\mu$ is recurrent for the transition kernel $\Lcal$.
Applying Corollary~\ref{cor:InvarianceMarkovDuale} to the transition kernel $\Lcal$ on $\Omega$ 
yields the same result for $P$: under the same hypotheses on $f$ and $g$, if
\begin{equation*}
 (I-P) f = g,
\end{equation*}
then:
\begin{equation*}
 (I-P_\Psi) f_{|\Psi} = g_{|\Psi}.
\end{equation*}

\subsection{From Markov chains to dynamical systems}
\label{subsec:MarkovToDynamique}

As any Markov chain can be encoded as a dynamical system, the converse is also true. 
As mentioned in the introduction of this section, a dynamical system can be seen as 
a Markov chain with transition kernel $P_x = \delta_{T(x)}$.

\smallskip

Let $(X_n)_{n \geq 0}$ be a Markov chain on a Polish state space $\Omega$, 
with transition kernel $P$ and recurrent stationary measure $\mu$. In this Subsection, 
we recall a classical result in potential theory (and a variant of~\cite[Corollary~1.11]{Revuz:1975}), 
and apply it to dynamical systems.

\begin{proposition}
\label{prop:InvarianceMarkov}
 
 Let $\Psi \subset \Omega$ be measurable, with $\mu (\Psi) > 0$. Let 
 $f$, $g \in \Lbb^\infty (\Omega, \mu)$ with $\Supp (g) \subset \Psi$.
 Assume that:
 \begin{equation*}
  (I-P) f = g.
 \end{equation*}
 Then:
 \begin{equation*}
  (I-P_\Psi) f_{|\Psi} = g_{|\Psi}.
 \end{equation*} 
\end{proposition}

\begin{proof}

 Let $\Psi$, $f$ and $g$ be as in the hypotheses of the proposition. 
 
 \smallskip
 
 We define the stopping time:
 \begin{equation*}
  \widetilde{T}_\Psi 
  := \inf \{n \geq 0: \ X_n \in \Psi\}.
 \end{equation*}
 Since the Markov chain is recurrent, for almost every $\omega$, either $\widetilde{T}_\Psi <+\infty$ 
 almost surely or $\widetilde{T}_\Psi = +\infty$ almost surely. Let $\Psi^\infty := \{\widetilde{T}_\Psi <+\infty \text{ almost surely}\}$. 
 Using again the recurrence of the Markov chain, we see that $\Psi \subset \Psi^\infty$ and 
 the set $\Psi^\infty$ is invariant. Without loss of generality, we restrict ourselves 
 to $\Psi^\infty$; that is, we assume that $\widetilde{T}_\Psi <+\infty$ almost surely.
 
 \smallskip
 
 Let $\widetilde{X}_n := X_{n \wedge \widetilde{T}_\Psi}$. Then $\widetilde{X}_n$ 
 is a Markov chain on $\Omega$, with transition kernel $\widetilde{P}$. 
 If $\omega \in \Psi$, then $\widetilde{P}_x = \delta_x$; otherwise, $\widetilde{P}_x = P_x$. 
 In addition, recall that $f = P(f)$ on $\Psi^c$. Hence,
 \begin{equation*}
  \widetilde{P} (f) 
  = \mathbf{1}_{\Psi} \widetilde{P} (f) + \mathbf{1}_{\Psi^c} \widetilde{P} (f) 
  = \mathbf{1}_{\Psi} f + \mathbf{1}_{\Psi^c} P(f) 
  = f.
 \end{equation*}
 In other words, for almost every $\omega \in \Omega$, we have $\Ebb_x (f(\widetilde{X}_1)) = \widetilde{X}_0$, 
 so $(f(\widetilde{X}_n))_{n \geq 0}$ is a martingale that is bounded almost surely. In addition, 
 it converges almost surely to $f(X_{\widetilde{T}_\Psi})$. By the dominated convergence theorem, 
 almost everywhere,
 \begin{equation}
  \label{eq:LimiteTemps0}
  \Ebb_\omega (f(X_{\widetilde{T}_\Psi})) 
  = \Ebb_\omega (f(X_0)) 
  = f(\omega).
 \end{equation}
 Let $\widetilde{P}_\Psi$ be the transition kernel of the Markov chain 
 $(X_0, X_{\widetilde{T}_\Psi}, X_{\widetilde{T}_\Psi}, \ldots)$, so that Equation~\eqref{eq:LimiteTemps0} 
 reads $\widetilde{P}_\Psi (f) = f$. Recall that $T_\Psi$ is the first positive time for which the Markov chain hits $\Psi$. 
 From the point of view of Markov operators, this means that $P_\Psi = P \widetilde{P}_\Psi$, and:
 \begin{equation*}
  P_\Psi (f) 
  = P (\widetilde{P}_\Psi (f))
  = P (f)
  = f-g.
 \end{equation*}
 In particular, this equation holds on $\Psi$, which is the conclusion of the proposition.
\end{proof}

%

Let $(A, \mathm, T)$ be a dynamical system, with $A$ Polish, and $\mathm$ a $\sigma$-finite, $T$-invariant, recurrent measure. 
The transformation $T$ gives rise to a recurrent transition kernel, which we will also denote by $T$:
\begin{equation*}
 T_x 
 := \delta_{T(x)},
\end{equation*}
such that the operator $T$ acting on $\Lbb^p (A, \mathm)$ is just the Koopman operator:
\begin{equation*}
 T (f) (x)
 = f \circ T (x).
\end{equation*}
Given any $B \subset A$ with positive measure, the transition kernel of the induced Markov chain 
is $T_B$. When applied to $T$, Proposition~\ref{prop:InvarianceMarkov} tells us that, 
for $f$, $g \in \Lbb^p (A, \mathm)$ with $g$ supported on $B$, if $f-f\circ T = g$, then $f-f \circ T_B = g$. 
This result was mentioned in Remark~\ref{rmk:Cobords}.

\smallskip

The transfer operator $\Lcal$ is then the Markov kernel dual to $T$. 
For instance, if $T$ is invertible, $\Lcal$ is the Koopman operator for $T^{-1}$; if $T$ has countably 
many branches, 
\begin{equation*}
 \Lcal (f) (x) 
 = \sum_{y \in T^{-1} (\{x\})} \frac{f(y)}{\Jac_{\mathm} (y)},
\end{equation*}
so the transition kernel corresponding to $\Lcal$ is:
\begin{equation*}
 \Lcal_x 
 = \sum_{y \in T^{-1} (\{x\})} \frac{1}{\Jac_{\mathm} (y)} \delta_y.
\end{equation*}
The measure $\mu$ is stationary for $\Lcal$ and, by Lemma~\ref{lem:RecurrenceDual}, 
it is also recurrent. Let $B$ have positive measure. Then $\Lcal_B$, defined as the dual of $T_B$, 
is also the Markov transition kernel induced by $\Lcal$ on $B$.
The case $p = \infty$ of Proposition~\ref{prop:InvarianceInduction} follows immediately.

\smallskip

Overall, the use of Lemma~\ref{lem:RecurrenceDual} makes this proof quite similar to, 
and not much shorter than, the direct proof of Proposition~\ref{prop:InvarianceInduction}. 
While it requires more background on Markov chains, seeing $\Lcal$ as the transition kernel 
of a stochastic process has the advantage of allowing the explicit use of backwards stopping 
times, which makes some manipulations easier.

\smallskip

Finally, let us recall the classical result asserting that solutions of the Poisson equation on $\Psi$ 
can be extended to solutions of the Poisson equation on $\Omega$, and highlight its relation with Lemma~\ref{lem:PoissonExtension}.

\begin{lemma}
\label{lem:ExtensionMarkov}
 
 Let $P$ be a transition kernel on $\Omega$ with invariant measure $\mu$. Assume that $\mu$ is $\sigma$-finite and recurrent. 
 Choose $\Psi \subset \Omega$ with $\mu (\Psi) >0$.
 
 \smallskip
 
 Let $f_\Psi \in \Lbb^\infty (\Psi, \mu)$, $g \in \Lbb^\infty (\Omega, \mu)$ be such that $g \equiv 0$ on $\Psi^c$ and:
 \begin{equation*}
 (I-P_\Psi) f_\Psi = g_{|\Psi}.
 \end{equation*}
 Then there is a unique function $f \in \Lbb^\infty (\Omega, \mu)$ such that $f_{|\Psi} = f_\Psi$ and:
 \begin{equation*}
 (I-P) f = g.
 \end{equation*}
 In addition, $\norm{f}{\Lbb^\infty (\Omega, \mu)} = \norm{f_\Psi}{\Lbb^\infty (\Psi, \mu)}$.
\end{lemma}

\begin{proof}
 
 Let $f_\Psi$ and $g$ be as in the hypotheses of the lemma.  As in the proof of Proposition~\ref{prop:InvarianceMarkov}, 
 let $\widetilde{T}$ be the stopping time $\inf \{n \geq 0: \ X_n \in \Psi\}$.
 For $\omega \in \Omega$, let:
 \begin{equation*}
  f (\omega) 
  := \Ebb_\omega (f(X_{\widetilde{T}})).
 \end{equation*}
 Note that $f_{|\Psi} \equiv f_\Psi$, and $\norm{f}{\Lbb^\infty (\Omega, \mu)} = \norm{f_\Psi}{\Lbb^\infty (\Psi, \mu)}$.
 
 \smallskip
 
 As in the proof of Proposition~\ref{prop:InvarianceMarkov}, we have $\widetilde{P}_\Psi (f) \equiv f$ almost everywhere. 
 On $\Psi^c$, we also have $P (f) \equiv \widetilde{P}_\Psi (f) \equiv f$, so $(I-P) (f) = 0 = g$. On $\Psi$, 
 \begin{equation*}
  P (f)_{|\Psi}
  = P (\widetilde{P}_\Psi (f))_{|\Psi}
  = P_\Psi (f_{|\Psi})
  = P_\Psi (f_\Psi)
  = f_{|\Psi}-g_{|\Psi}.
 \end{equation*}
 Hence, $(I-P) (f) = g$ almost everywhere.
 
 \smallskip
 
 Uniqueness follows from the same argument as in Lemma~\ref{lem:Liouville}, for instance by encoding the Markov chain.
\end{proof}

If applied to the Markov chain with transition kernel $\Lcal$ on $A$, Lemma~\ref{lem:ExtensionMarkov} 
immediately yields Lemma~\ref{lem:PoissonExtension}.

%

\section{Poisson equation and Green-Kubo formula}
\label{sec:GreenKubo}

Given an ergodic and recurrent dynamical system $(A, \mathm, T)$ and a function $f : A \to \R$, 
a set $B\subset A$ with positive measure, define the function $\Sigma_Bf :B\to\R$ as the sum 
of the values of $f$ over the excursion out of $B$:
\begin{equation*}
 \Sigma_B (f) (x) 
 := \sum_{k=0}^{\varphi_B (x)-1} f(T^k (x))\, .
\end{equation*}

Let $f \in \Lbb^1 (A, \mathm)$. Then $\Sigma_B(f) \in \Lbb^1 (B, \mathm_{|B})$, with integral 
\begin{equation}
\label{eq:InvarianceIntegrale}
 \int_B \Sigma_B (f) \dd \mathm 
 = \int_A f \dd \mathm.
\end{equation}
This equation is a classical extension of Kac's formula, which admits multiple proofs. 
The most direct one proceeds along the lines of Proposition~\ref{prop:InvarianceInduction}. 
It can also be proved by looking at the limit behaviour of Birkhoff sums, which shall be 
the object of Section~\ref{sec:Distributions}. Finally, in some cases, it can be recovered by 
using a well-chosen coboundary. For $x \in A$, let:
\begin{equation}
 \label{eq:DefinitionCobord}
 C(f) (x) 
 := \left\{ 
 \begin{array}{lll}
  0 & \text{ if } & x \in B \\
  \sum_{k=0}^{\varphi_B (x)-1} f \circ T^k (x) & \text{ if } & x \notin B
 \end{array}
 \right. .
\end{equation}
Then 
\begin{equation}
\label{eq:Cohomologue}
\Sigma_B (f) \mathbf{1}_B - f = C(f) \circ T - C(f)\, .
\end{equation}
In particular, if $C(f) \in \Lbb^1 (A, \mathm)$, 
integrating over $A$ yields Equation~\eqref{eq:InvarianceIntegrale}.

\smallskip

Let $f \in \Lbb^1 (A, \mathm) \cap \Lbb^2 (A, \mathm)$ with null $\mathm$-integral. 
We set, when it is well defined, the Green-Kubo formula:
\begin{align*}
 \sigma^2_{GK} (A, \mathm, T; f,f) 
 & := \int_A f^2 \dd \mathm + 2 \sum_{n \geq 1} \int_A f \cdot f\circ T^n \dd \mathm \\
 & = \int_A f^2 \dd \mathm + 2 \sum_{n \geq 1} \int_A f \cdot \Lcal^n (f) \dd \mathm.
\end{align*}
These quantities $ \sigma^2_{GK} (A, \mathm, T; f,f)$ appear in limit theorems whether $\mathm$ is finite 
\cite[Equation~2.21]{Nagaev:1957} or infinite~\cite[Theorem~2.4]{PeneThomine:2017}. 
By~\cite[Corollary~1.13 and Appendix]{PeneThomine:2017}, the Green-Kubo formula, in some cases, 
satisfies the same kind of invariance under induction as the integral~:
\begin{equation*}
 \sigma^2_{GK} (A, \mathm, T; f,f) 
 = \sigma^2_{GK} (B, \mathm_{|B}, T_B ; \Sigma_B (f),\Sigma_B (f))\, ,
\end{equation*}
making it a degree $2$ analog of Equation~\eqref{eq:InvarianceIntegrale}. 
This observation was central in~\cite{PeneThomine:2017}, as it allowed the estimation 
of hitting probabilities in a diffusion model. 
The goal of this section is to prove rigorously this equality in some cases, 
and show the relation between the invariance under induction of 
$\sigma_{GK}^2$ and Proposition~\ref{prop:InvarianceInduction}.
We shall study its associated symmetric bilinear form:
\begin{equation*}
 \sigma^2_{GK} (A, \mathm, T; f,g) 
 = \int_A fg \dd \mathm + \sum_{n \geq 1} \int_A f \Lcal^n (g) \dd \mathm + \sum_{n \geq 1} \int_A g \Lcal^n (f) \dd \mathm\, .
\end{equation*}
We shall also use a notion of mixing:

\begin{definition}[Mixing of functions]
 
 Let $f,g$ be two functions defined on $(A,\mathm)$. We say that 
 the dynamical system $(A,\mathm,T)$ \emph{mixes} $(f,g)$ if the quantities 
 $\int_A g\cdot \Lcal^n f\dd\mathm = \int_Ag\circ T^n\cdot  f\dd\mathm$ and 
 $\int_Af\cdot \Lcal^n g\dd\mathm = \int_Af\circ T^n\cdot  g\dd\mathm$ are well-defined and 
 converge to $\alpha(f,g) = \alpha (g,f)$ as $n$ goes to infinity.
\end{definition}

For instance, if $\mathm$ is a probability measure and $(A,\mathm,T)$ is mixing, then 
$(A,\mathm,T)$ mixes any pair of square-integrable functions, with $\alpha(f,g) = \int_A f \dd \mathm \cdot \int_A g \dd \mathm$.

\smallskip

The definition of Green-Kubo's formula can be widened, for instance by allowing 
the infinite sums to converge only in Abel's sense. With such a modification, the mixing conditions 
in this section can be lifted. However, working with absolutely converging series will be helpful 
in Section~\ref{sec:Degre3}.

\begin{theorem}
\label{theo:InvarianceGreenKubo}

Let  $(A,\mathm,T)$ be a recurrent, ergodic, measure preserving dynamical system, 
with $\mathm$ a $\sigma$-finite measure and $B\subset A$ with $0< \mathm(B) \leq +\infty$.
Let $p_1$, $p_2\in[1,+\infty]$ be such that $\frac{1}{p_1}+\frac{1}{p_2}=1$.
Let $f_1$, $f_2$ be two integrable functions defined on $A$ with null integral, 
such that $\sigma^2_{GK} (A, \mathm, T; f_1,f_2)$ is well defined and such that
%
\begin{equation}
 \sum_{n\geq 0} \norm{\mathbf{1}_B \Lcal^n\left(\Sigma_B(f_i)\mathbf{1}_B\right)}{\Lbb^{p_i}(A,\mathm)}
 <+\infty
\end{equation}
and
\begin{equation}
 \label{eq:hypB}
 \sum_{n\ge 0} \norm{\Lcal_B^n\Sigma_B(f_i)}{\Lbb^{p_i}(B,\mathm_{|B})}
 <+\infty\, ,
\end{equation}
for $i=1$, $2$. Assume moreover that $(A,\mathm,T)$ mixes $(C(f_1),f_2)$ and $(\mathbf{1}_B\Sigma_B(f_1),C(f_2))$.
Then the following quantities are well defined and equal:
\begin{equation}
 \label{invarianceinduction}
 \sigma^2_{GK} (A, \mathm, T; f_1,f_2) 
 = \sigma^2_{GK} (B, \mathm_{|B}, T_B ; \Sigma_B (f_1), \Sigma_B (f_2))\, .
\end{equation}
\end{theorem}

If $f_1$ or $f_2$ has non zero integral, then an error term may appear.

\begin{remark}
\label{rem:InvarianceGreenKubo}

Assume that the conclusion~\eqref{invarianceinduction} of Theorem~\ref{theo:InvarianceGreenKubo} holds true, 
and that the quantity $\sigma^2_{GK} (B, \mathm_{|B}, T_B ; \varphi_B, \Sigma_B (f_2))$ is well defined. 
Then, for every $c\in\R$,
\begin{align*}
 \sigma^2_{GK} (A, \mathm, T; f_1+c,f_2) 
 & = \sigma^2_{GK} (A, \mathm, T; f_1,f_2) \\
 & = \sigma^2_{GK} (B, \mathm_{|B}, T_B ; \Sigma_B (f_1), \Sigma_B (f_2))\\
 & = \sigma^2_{GK} (B, \mathm_{|B}, T_B ; \Sigma_B (f_1+c), \Sigma_B (f_2)) - c \,\sigma^2_{GK} (B, \mathm_{|B}, T_B ; \varphi_B, \Sigma_B (f_2))\, .
\end{align*}
\end{remark}

The rest of this section is devoted to the proof of Theorem~\ref{theo:InvarianceGreenKubo},
which is divided in two parts corresponding to the two following sections:
\begin{itemize}
\item invariance of $\sigma^2_{GK}$ under addition of coboundaries. Then, by Equation~\eqref{eq:Cohomologue}, 
a function $f$ is cohomologous to $\Sigma_B (f) \mathbf{1}_B$, ensuring that 
\begin{equation*}
 \sigma^2_{GK} (A, \mathm, T; f_1,f_2) 
 = \sigma^2_{GK} (A, \mathm, T; \Sigma_B(f_1)\mathbf{1}_B,\Sigma_B(f_2)\mathbf{1}_B)\, .
\end{equation*}
\item invariance under induction for functions supported on $B$, ensuring that
\begin{equation*}
 \sigma^2_{GK} (A, \mathm, T;\Sigma_B (f_1) \mathbf{1}_B, \Sigma_B (f_2) \mathbf{1}_B)
 =\sigma^2_{GK} (B, \mathm_{|B}, T_B; \Sigma_B (f_1), \Sigma_B (f_2))\, ;
\end{equation*}
for this step we will use Proposition~\ref{prop:InvarianceInduction} together with Lemma~\ref{lem:Liouville}, 
yielding 
\begin{equation*}
 \int_B \Sigma_B(f_i) \sum_{n\ge 0} \Lcal^n\left(\Sigma_B(f_j)\mathbf{1}_B\right)\dd\mathm
 = \int_B \Sigma_B(f_i) \sum_{n\geq 0}\Lcal_B^n \left(\Sigma_B(f_j)\right)\dd\mathm \, .
\end{equation*}
\end{itemize}

\subsection{Invariance of \texorpdfstring{$\sigma_{GK}^2$}{Green-Kubo's formula} under addition of coboundaries}

The goal of this subsection is to prove that under the mixing assumptions of Theorem~\ref{theo:InvarianceGreenKubo},
\begin{equation}
 \label{invcob}
 \sigma^2_{GK} (A, \mathm, T; f_1,f_2) 
 = \sigma^2_{GK} (A, \mathm, T; \Sigma_B (f_1) \mathbf{1}_B, \Sigma_B (f_2) \mathbf{1}_B).
\end{equation}
To this end we prove the following general proposition.

\begin{proposition}
\label{prop:GKCobord3}
Let $(A,\mathm,T)$ be a measure preserving dynamical system, with $\mathm$ a $\sigma$-finite measure.
Let $f$, $g$ be two functions defined on $A$ such that $\sigma^2_{GK} (A, \mathm, T; f,g) $ 
is well defined. Let $u$ be a function defined on $A$ such that
$(A,\mathm,T)$ mixes $(u,g)$. Then $\sigma^2_{GK} (A, \mathm, T; f+u\circ T-u,g)$
is well defined and
\begin{equation}
 \sigma^2_{GK} (A, \mathm, T; f,g)
 = \sigma^2_{GK} (A, \mathm, T; f+u\circ T-u,g)\, .
\end{equation}
\end{proposition}

Note that we do not assume in this result that the functions have null integral.
This result will appear as a direct consequence of the two following lemmas.

\begin{lemma}
\label{lem:GKCobord1}

Let $(A,\mathm,T)$ be a measure preserving dynamical system, with $\mathm$ a $\sigma$-finite measure.
Let $g$ and $u$ be two functions defined on $A$ such that $\int_A g \cdot \Lcal^n u\dd \mathm=\int_Ag\circ T^n \cdot u\dd \mathm$ 
is well defined for any $n\ge 0$ and converges to $\alpha(g,u)$ when $n$ goes to infinity. Then
\begin{equation}
 \label{eq:GKCobord1}
 \sum_{n \geq 0} \int_A g \Lcal^n (u \circ T-u) \dd \mathm 
 = \int_A g \cdot u \circ T \dd \mathm-\alpha(g,u).
\end{equation}
\end{lemma}

\begin{proof}
 
 Let $N \geq 0$. Then:
 \begin{align*}
  \sum_{n=0}^N \int_A g \Lcal^n (u \circ T-u) \dd \mathm 
  & = \sum_{n=0}^N \int_A g\circ T^n \cdot (u \circ T-u) \dd \mathm \\
  & = \int_A g (u \circ T-u) \dd \mathm + \sum_{n=1}^N \left( \int_A g \circ T^{n-1} \cdot u \dd \mathm - \int_A g \circ T^n \cdot u \dd \mathm  \right) \\ 
  & = \int_A g \cdot u \circ T \dd \mathm - \int_A g \circ T^N \cdot u \dd \mathm \\ 
  & = \int_A g \cdot u \circ T \dd \mathm - \int_A g \Lcal^N (u) \dd \mathm.
 \end{align*}
 Taking the limit as $N$ goes to infinity yields Equation~\eqref{eq:GKCobord1}. 
\end{proof}

\begin{lemma}
\label{lem:GKCobord2}
Let $(A,\mathm,T)$ be a measure preserving dynamical system, with $\mathm$ a $\sigma$-finite measure.
Let $g$ and $u$ be two functions defined on $A$ such that $\int_Au \cdot \Lcal^n g\dd \mathm=\int_1u\circ T^n \cdot g\dd\mathm$ is well defined for any $n\geq 1$ and
converges to $\alpha(g,u)$ when $n$ goes to infinity. Then 
 \begin{equation}
 \label{eq:GKCobord2}
  \sum_{n \geq 0} \int_A (u \circ T-u) \Lcal^n (g) \dd \mathm 
  = \alpha(g,u) - \int_A u g \dd \mathm \, .
 \end{equation}
\end{lemma}

\begin{proof}
 
 Let $N \geq 0$. Then:
 \begin{align*}
  \sum_{n=0}^N \int_A (u \circ T-u) \Lcal^n (g) \dd \mathm 
  & = \sum_{n=0}^N \int_A (u \circ T^{n+1}-u \circ T^n) g \dd \mathm \\
  & = \int_A u \circ T^{N+1} \cdot g \dd \mathm - \int_A u g \dd \mathm \\ 
  & = \int_A u \cdot \Lcal^{N+1} (g) \dd \mathm - \int_A u g \dd \mathm.
 \end{align*}
 Taking the limit as $N$ goes to infinity yields Equation~\eqref{eq:GKCobord2}. 
\end{proof}

\begin{proof}[Proof of Proposition~\ref{prop:GKCobord3}]
 
 By bilinearity of $\sigma_{GK}^2$, it is enough to prove that $\sigma^2_{GK} (A, \mathm, T; f,u \circ T-u)$ is well defined and null, 
 or in other words that
 \begin{equation*}
 -\int_A g (u \circ T - u) \dd \mathm 
   + \sum_{n \geq 0} \int_A g \Lcal^n (u \circ T - u) \dd \mathm 
   + \sum_{n \geq 0} \int_A (u \circ T - u) \Lcal^n (g) \dd \mathm
   =0 \, ,
 \end{equation*}
which comes from Lemmas~\ref{lem:GKCobord1} and~\ref{lem:GKCobord2}.
\end{proof}

\subsection{Invariance of \texorpdfstring{$\sigma_{GK}^2$}{Green-Kubo's formula} under induction}

We now focus on the second part of the proof of Theorem~\ref{theo:InvarianceGreenKubo}. We want to show 
that, under the hypotheses of this theorem,
\begin{equation*}
 \sigma^2_{GK} (A, \mathm, T;\Sigma_B (f_1) \mathbf{1}_B, \Sigma_B (f_2) \mathbf{1}_B)
 =\sigma^2_{GK} (B, \mathm_{|B}, T_B; \Sigma_B (f_1), \Sigma_B (f_2))\, .
\end{equation*}
This follows from the next lemma:

\begin{lemma}
\label{lem:sigma2induct}

Let $(A,\mathm,T)$ be a recurrent, ergodic, measure preserving dynamical system, with $\mathm$ a $\sigma$-finite measure.
Let $B\subset A$ with $0<\mathm(B) \leq + \infty$. 
Let $p_1$, $p_2\in[1,+\infty]$ be such that $\frac{1}{p_1}+\frac{1}{p_2}=1$.
Let $g_1$, $g_2$ be two integrable functions with null integral defined on $B$ such that
\begin{equation}
 \label{sum1}
 \sum_{n\ge 0} \norm{\Lcal^n\left(g_i\mathbf{1}_B\right)}{\Lbb^{p_i}(A,\mathm)}
 <+\infty\, ,
\end{equation}
and
\begin{equation}
 \label{sum2}
 \sum_{n\ge 0} \norm{\Lcal_B^n g_i}{\Lbb^{p_i}(B,\mathm_{|B})}
 <+\infty\, ,
\end{equation}
for $i=1,2$. Then the following quantities are well defined and equal:
\begin{equation}
 \label{invinduc1}
 \sigma^2_{GK} (A, \mathm, T;g_1 \mathbf{1}_B, g_2 \mathbf{1}_B)
 = \sigma^2_{GK} (B, \mathm_{|B}, T_B; g_1, g_2)\, .
\end{equation}
\end{lemma}

\begin{proof}
Let $G_i := \sum_{n=0}^{+ \infty} \Lcal^n (g_i\mathbf{1}_B)$, which is well-defined 
and in $\Lbb^{p_i}(A,\mathm)$ by the condition~\eqref{sum1}. Then $(I-\Lcal) G_i = g_i \mathbf{1}_B$. 
By Proposition~\ref{prop:InvarianceInduction}, since $g_i\mathbf{1}_B$ is in $\Lbb^{p_i}(A,\mathm)$, 
we have $(I-\Lcal_B) (G_i)_{|B} = g_i$. 

\smallskip

The condition~\eqref{sum2} implies:
\begin{equation*}
 g_i
 = (I - \Lcal_B) \sum_{n \geq 0} \Lcal_B^n g_i \text{ in } \Lbb^{p_i}(B,\mathm)\, .
\end{equation*}
Since the system is assumed to be ergodic, by Lemma~\ref{lem:Liouville}, there exists a constant $C_0(G_i)$ such that
\begin{equation*}
(G_i)_{|B} = \sum_{n=0}^{+ \infty} \Lcal_B^n g_i+C_0(G_i)\, .
\end{equation*} 
Integrating against $g_j \mathbf{1}_B$, which has integral zero, we get:
\begin{align*}
 \sum_{n \geq 0} \int_A (g_j \mathbf{1}_B) \Lcal^n (g_i \mathbf{1}_B) \dd \mathm 
 & = \int_A (g_j \mathbf{1}_B) \sum_{n \geq 0} \Lcal^n (g_i \mathbf{1}_B) \dd \mathm \\
 & = \int_B g_j \, G_i \dd \mathm \\
 & = \int_B g_j \sum_{n=0}^{+ \infty} \Lcal_B^n g_i \dd \mathm \\
 & = \sum_{n=0}^{+ \infty} \int_B g_j \Lcal_B^n g_i \dd \mathm.
\end{align*}
Finally
\begin{align*}
 \sigma^2_{GK} (A, \mathm, T; g_1 \mathbf{1}_B, g_2 \mathbf{1}_B) 
 & = - \int_B g_1 g_2 \dd \mathm + \sum_{n \geq 0} \int_B g_1 \Lcal^n (g_2 \mathbf{1}_B) \dd \mathm + \sum_{n \geq 0} \int_B g_2 \Lcal^n (g_1 \mathbf{1}_B) \dd \mathm \\
 & = - \int_Bg_1g_2 \dd \mathm + \sum_{n \geq 0} \int_B g_1 \Lcal_B^n (g_2) \dd \mathm + \sum_{n \geq 0} \int_B g_2 \Lcal_B^n g_1 \dd \mathm \\
 & = \sigma^2_{GK} (B, \mathm_{|B}, T_B; g_1, g_2). \qedhere
\end{align*}
\end{proof}

We can now finish the proof of Theorem~\ref{theo:InvarianceGreenKubo}.

\begin{proof}[Proof of Theorem~\ref{theo:InvarianceGreenKubo}]

The quantity $\sigma^2_{GK} (A, \mathm, T; f_1,f_2)$ is well-defined by hypothesis. 
Since $(A,\mathm,T)$ mixes $(C(f_1),f_2)$, by Proposition~\ref{prop:GKCobord3}, 
the following quantities are well-defined and equal:
\begin{equation*}
 \sigma^2_{GK} (A, \mathm, T; f_1,f_2)
 = \sigma^2_{GK} (A, \mathm, T; \Sigma_B(f_1)\mathbf{1}_B,f_2) \, .
\end{equation*}
Since $(A,\mathm,T)$ mixes $(\mathbf{1}_B\Sigma_B(f_1),C(f_2))$, 
applying again Proposition~\ref{prop:GKCobord3} yields:
\begin{equation*}
 \sigma^2_{GK} (A, \mathm, T; f_1,f_2)
 \sigma^2_{GK} (A, \mathm, T; \Sigma_B(f_1)\mathbf{1}_B,\Sigma_B(f_2)\mathbf{1}_B) \, ,
\end{equation*}
that is, Equation~\eqref{invcob}.

\smallskip

Due to the condition~\eqref{eq:hypB}, the quantity $\sigma^2_{GK} (B, \mathm_{|B}, T_B ; \Sigma_B (f_1), \Sigma_B (f_2))$ is well-defined.
Lemma~\ref{lem:sigma2induct} applied to $g_i:=\Sigma_B (f_i)$ ensures that 
the following quantities are well defined and equal:
\begin{align*}
 \sigma^2_{GK} (A, \mathm, T; f_1,f_2) 
 & = \sigma^2_{GK} (A, \mathm, T;\Sigma_B (f_1) \mathbf{1}_B, \Sigma_B (f_2) \mathbf{1}_B) \\
 & = \sigma^2_{GK} (B, \mathm_{|B}, T_B; \Sigma_B (f_1), \Sigma_B (f_2))\, . \qedhere
\end{align*}
\end{proof}

\section{A degree \texorpdfstring{$3$}{3} invariant}
\label{sec:Degre3}

As an application of these methods, we shall now study a degree $3$ analog of the Green-Kubo formula. 

\subsection{Definition and property for \texorpdfstring{$\tau^3$}{Tau 3}}

Informally, $ \sigma_{GK}^2 (A, \mathm, T ; f,g) $ can be rewritten
\begin{equation*}
 \sigma_{GK}^2 (A, \mathm, T ; f,g) 
 = \sum_{n \in \Z} \int_A f \cdot g \circ T^n \dd \mathm,
\end{equation*}
which can be made rigorous for many examples of invertible hyperbolic dynamical systems. 
A degree $3$ analog would then be, for suitable functions $f$, $g$, $h$:
\begin{equation}
 \label{eq:EquationTauNaive}
 \sum_{n, m \in \Z} \int_A f \cdot g \circ T^n \cdot h \circ T^m \dd \mathm\, ,
\end{equation}
as soon as the sum is absolutely convergent. To make sense of this formula when $T$ is not invertible, note that:
\begin{equation*}
 \int_A f \cdot g \circ T^n \cdot h \circ T^m \dd \mathm 
 = \int_A f \circ T^{-\min \{n,m,0\}} \cdot g \circ T^{n-\min \{n,m,0\}} \cdot h \circ T^{m-\min \{n,m,0\}} \dd \mathm.
\end{equation*}

In what follows, we shall write $\sum_{\Alt}$ for the sum over all permutations of $\{f,g,h\}$ in the formula. 
For instance, $\sum_{\Alt} fgh = 6fgh$. 
A careful reorganization of the sum above yields the following formula.

\begin{definition}
\label{def:Tau3}
 
 Let $(A, \mathm, T)$ be a measure preserving dynamical system.
 Define:
 \begin{align}
  \tau^3 (A, \mathm, T ; f, g, h) 
  & := \sum_{\Alt} \left[ \frac{1}{6} \int_A fgh \dd \mathm   + \frac{1}{2}\sum_{n \geq 1} \int_A fg \cdot \Lcal^n (h) \dd \mathm \right. \label{eq:DefinitionTau} \\
  &  \hspace{4em} \left. + \frac{1}{2} \sum_{n \geq 1} \int_A f \cdot \Lcal^n (gh) \dd \mathm + \sum_{n, m \geq 1} \int_A f \cdot \Lcal^n (g \Lcal^m(h)) \dd \mathm \right]\, , \nonumber
 \end{align}
whenever the sums are absolutely convergent.
\end{definition}

The following criterion can be used to check that the sum in Equation~\eqref{eq:DefinitionTau} 
is indeed convergent.

\begin{lemma}
Let $(A, \mathm, T)$ be a measure preserving dynamical system, with $\mathm$
$\sigma$-finite. Let $\Bcal$ be a Banach space of functions continuously embedded in $\Lbb^1_0(A, \mathm) \cap \Lbb^3(A, \mathm)$, 
and on which $(\Lcal^n)_{n \geq 0}$ is summable. Assume moreover that
\begin{equation}
 \label{sum1+2}
 K_0
 := \sup_{\substack{f,g,h\in\Bcal \\ \norm{f}{\Bcal},\norm{g}{\Bcal},\norm{h}{\Bcal} \leq 1}} \sum_{n\ge 0}\left|\int_A f\Lcal^n (gh)\dd\mathm\right|
 <+\infty\, .
\end{equation}
%
Then the sum in Definition~\ref{def:Tau3} converges absolutely, so that $\tau^3$ is a well-defined trilinear symmetric form on $\Bcal$.
\end{lemma}

Observe that the assumption~\eqref{sum1+2} holds true if there exists a function $H$ defined on $A$ with unit integral 
such that $\sum_{n\geq 0} \left|\int_A f\Lcal^n H\dd \mathm\right|<+\infty$ for every $f\in \Bcal$ (this is true for example 
for $H=\mathbf{1}_A/\mathm(A)$ if $\mathm$ is finite) and such that $(f,g)\mapsto fg-H\int_Afg \dd\mathm$ is continuous from
$\Bcal^2$ to a Banach space $\Bcal'$ continuously embedded in $\Lbb_0^1(A,\mathm)\cap\Lbb^{\frac{3}{2}}(A,\mathm)$ 
and on which $\Lcal^n$ is summable.

\smallskip

Another strategy to prove \eqref{sum1+2} could be to prove the existence of a sequence $(c_n)_n$ of real numbers such that
\begin{equation}
\sup_{\substack{f,g,h\in\Bcal \\ \norm{f}{\Bcal},\norm{g}{\Bcal},\norm{h}{\Bcal} \leq 1}} \sum_{n\geq 0} \left|\int_A f\left(\Lcal^n (gh)-c_n\int_A gh\dd \mathm\right)\dd\mathm \right|
<+\infty\, .
\end{equation}
When $\mathm$ is infinite, such an estimate is related to higher order terms in mixing estimates, 
$(c_n)_n$ being the speed of mixing (see for example~\cite{Pene:2019}).

\begin{proof}
 
 The sum contains four parts. The first part contains a single term, 
 which is finite since $\Bcal \subset \Lbb^3 (A, \mathm)$. 
For the second term, we observe that
 \begin{align*}
  \left| \int_A fg \cdot \Lcal^n (h) \dd \mathm \right|
  & \leq \norm{fg}{\Lbb^{3/2}} \norm{\Lcal^n (h)}{\Lbb^3} \\
  & \leq \norm{f}{\Lbb^3} \norm{g}{\Lbb^3} \norm{h}{\Bcal} \norm{\Lcal^n}{\Bcal \to \Lbb^3},
 \end{align*}
 which by hypothesis is summable in $n$.
 
 \smallskip
 
 Consider the third term. By hypothesis,
 %
 \begin{equation*}
  \sum_{n\ge 0}\left| \int_A f \cdot \Lcal^n (gh) \dd \mathm \right|
   \leq K_0\norm{f}{\Bcal}\norm{g}{\Bcal}\norm{h}{\Bcal}\, ,
 \end{equation*}
 which is finite.
 
 \smallskip
 
 Finally, let us focus on the fourth term. We have:
 \begin{equation*} 
 \sum_{n\ge 0} \left| \int_A f \cdot \Lcal^n (g \Lcal^m(h)) \dd \mathm \right|  
 \leq K_0 \norm{f}{\Bcal}  \norm{g }{\Bcal}\norm{\Lcal^m}{\Bcal \to \Bcal}\norm{h}{\Bcal}\,  ,
 \end{equation*}
 %
 %
 which is summable in $m$.
\end{proof}

The function $\tau^3$ admits a few equivalent definitions, which may be more convenient depending on the computations. 
First, making the sums start at $0$ instead of $1$, we get the next lemma (beware the change of signs).

\begin{lemma}
\label{lem:FormulaTau3}
It $\tau^3 (A, \mathm, T ; f, g, h)$ is absolutely convergent, then
\begin{align*}
 \tau^3 (A, \mathm, T ; f, g, h) 
 & = \sum_{\Alt} \left[ \frac{1}{6} \int_A fgh \dd \mathm - \frac{1}{2}\sum_{n \geq 0} \int_A fg \cdot \Lcal^n (h) \dd \mathm \right. \\
 &  \hspace{4em} \left. - \frac{1}{2} \sum_{n \geq 0} \int_A f \cdot \Lcal^n (gh) \dd \mathm + \sum_{n, m \geq 0} \int_A f \cdot \Lcal^n (g \Lcal^m(h)) \dd \mathm \right].
\end{align*}
\end{lemma}

\begin{proof}

Let us introduce the following notations:
\begin{align*}
 A_0 (f,g,h) &:=\int_A fgh\dd \mathm\, ,\quad
 & A_{1,i} (f,g,h) &:=\sum_{n\ge i}\int_Afg \cdot \Lcal^n (h) \dd \mathm\, ,\\
 A_{2,i} (f,g,h) &:=\sum_{n \geq i} \int_A f \cdot \Lcal^n (gh) \dd \mathm\, ,\quad
 & A_{3,i,j} (f,g,h) &:= \sum_{n\geq i, m \geq j} \int_A f \cdot \Lcal^n (g \Lcal^m(h)) \dd \mathm\, .
\end{align*}
Note that
\begin{align*}
 A_{1,1} & =A_{1,0}-A_0\, ,\quad
 A_{2,1} =A_{2,0}-A_0\, ,\\
 A_{3,1,1}&=A_{3,0,1}-A_{1,1}=A_{3,0,0}-A_{2,0}-A_{1,0}+A_0\, .
\end{align*}
Therefore
\begin{align*}
\tau^3 
 & = \sum_{\Alt}\left(\frac{A_0}6+\frac{A_{1,1}}2+\frac{A_{2,1}}2+A_{3,1,1}\right)=\sum_{\Alt}\left(\frac{A_0}6+\frac{A_{1,0}}2+\frac{A_{2,0}}2+A_{3,0,0}-A_{2,0}-A_{1,0}\right)\\
 & = \sum_{\Alt}\left(\frac{A_0}6-\frac{A_{1,0}}2-\frac{A_{2,0}}2+A_{3,0,0}\right)\, . \qedhere
\end{align*}
\end{proof}

Another useful formula follows by fixing the value of $n$ in Equation~\eqref{eq:EquationTauNaive}, 
and summing over $m$, which leads to the following result expressing $\tau^3$ in terms of $\sigma_{GK}^2$.

\begin{lemma}
\label{lem:tau3sigma2}
It $\tau^3 (A, \mathm, T ; f, g, h)$ is absolutely convergent, then
\begin{align}
 \tau^3 (A, \mathm, T ; f, g, h) 
 & = \sigma_{GK}^2 (A, \mathm, T ; fg, h) + \sum_{n \geq 1} \sigma_{GK}^2 (A, \mathm, T ; f \cdot g \circ T^n, h) \label{tau3sigma2} \\
 & \hspace{4em} + \sum_{n \geq 1} \sigma_{GK}^2 (A, \mathm, T ; f\circ T^n \cdot g, h). \nonumber
\end{align}

\end{lemma}
\begin{proof}
The absolute convergence of $\tau^3$ implies the absolute convergence of the sums 
in the right hand side of Equation~\eqref{tau3sigma2}, which is equal to:
\begin{align*}
&\int_A fgh \dd \mathm + \sum_{m \geq 1} \int_A fg\cdot \Lcal^m (h) \dd \mathm + \sum_{m \geq 1} \int_A h\cdot \Lcal^m (fg) \dd \mathm+B(f,g,h)+B(g,f,h)\\
&=A_0(f,g,h)+A_{1,1}(f,g,h)+A_{2,1}(h,f,g)+B(f,g,h)+B(g,f,h)\\
&=A_0(f,g,h)+\frac{A_{1,1}(f,g,h)+A_{1,1}(g,f,h)+A_{2,1}(h,f,g)+A_{2,1}(h,g,f)}2+B(f,g,h)+B(g,f,h)\, ,
\end{align*}
where the $A_i$ are defined in the proof of Lemma~\ref{lem:FormulaTau3}, and:
\begin{align*}
B(f,g,h)&:=
\sum_{n \geq 1}\left( \int_A f \cdot g \circ T^nh \dd \mathm + \sum_{m \geq 1} \int_A f \cdot g \circ T^n \Lcal^m (h) \dd \mathm + \sum_{m \geq 1} \int_A h \Lcal^m (f \cdot g \circ T^n) \dd \mathm\right)\\
& = \sum_{n \geq 1}\int_A \Lcal^n(fh) \cdot g  \dd \mathm +\sum_{n,m \geq 1} \int_Ag\cdot\Lcal^n(f \Lcal^m (h))  \dd \mathm \\
& \hspace{4em} + \sum_{1\le n<m} \int_A h \Lcal^{m-n} (g\Lcal^nf) \dd \mathm + \sum_{1\le m\le n} \int_A g\Lcal^{n-m}(h\Lcal^mf)\dd \mathm\\
&= A_{2,1}(g,f,h)+A_{3,1,1}(g,f,h)+A_{3,1,1}(h,g,f)+A_{3,1,1}(g,h,f)+A_{1,1}(h,g,f)\\
& = \frac{A_{2,1}(g,f,h)+A_{2,1}(g,h,f)+A_{1,1}(h,g,f)+A_{1,1}(g,h,f)}2\\
& \hspace{4em} + A_{3,1,1}(g,f,h)+A_{3,1,1}(h,g,f)+A_{3,1,1}(g,h,f)\, .
\end{align*}
Putting all the terms together ends the proof of the lemma.
\end{proof}

\subsection{Quasi-invariance under induction}

Given a measure preserving dynamical system $(A,\mathm,T)$ and a subset $B$ of $A$, 
we wish to study the relation between
\begin{equation*}
 \tau^3 (A, \mathm, T; f_1,f_2,f_3) 
 \text{ and } 
 \tau^3 (B, \mathm_{|B}, T_B ; \Sigma_B (f_1), \Sigma_B (f_2),  \Sigma_B (f_3))\, ,
\end{equation*}
for functions $f_i$ on $A$ with null integral, as we did with Theorem~\ref{theo:InvarianceGreenKubo} 
for $\sigma_{GK}^2 (A, \mathm, T; f_1,f_2)$ and $\sigma_{GK}^2 (B, \mathm_{|B}, T_B ; \Sigma_B (f_1), \Sigma_B (f_2))$.

\smallskip

In what follows, we shall use two Banach spaces:
\begin{itemize}
 \item a Banach space $\Bcal_1 \subset \Lbb^1_0(B,\mathm_{|B})\cap\Lbb^3(B,\mathm_{|B})$,
 \item a Banach space $\Bcal_2 \subset \Lbb^1_0(B,\mathm_{|B})\cap\Lbb^\frac{3}{2} (B,\mathm_{|B})$,
\end{itemize}
and a function $H\in \Lbb^1 (B, \mathm_{|B})$ such that $\int_B H\dd\mathm=1$. 
Letting
\begin{equation}
 \label{eq:DefinitionProduitBoite}
 f\boxtimes g
 := fg-H\int_Bfg\dd\mathm\, ,
\end{equation}
we shall furthermore assume that $(f,g) \mapsto f\boxtimes g$ is continuous from $\Bcal_1^2$ to $\Bcal_2$. 
If $\mathm(B) <+\infty$, the simplest -- but not the sole -- choice is to take $H := \mathm (B)^{-1} \mathbf{1}_B$.

\begin{theorem}
\label{theo:Invariancetau3}

Let  $(A,\mathm,T)$ be a recurrent, ergodic, measure preserving dynamical system, 
with $\mathm$ a $\sigma$-finite measure. Let $B\subset A$ with $0 < \mathm(B) \leq +\infty$. 
Let $\Bcal_1$, $\Bcal_2$ and $H$ be as above.

\smallskip

Assume that
\begin{equation*}
 \norm{\Lcal_B^n}{\Bcal_1\rightarrow \Bcal_1}\, , 
 \hspace{2em} \norm{\mathbf{1}_B\Lcal^n}{\Bcal_2\rightarrow \Lbb^{\frac{3}{2}}(B,\mathm_{|B})}\, ,
 \hspace{2em} \norm{\Lcal_B^n}{\Bcal_2\rightarrow \Lbb^{\frac{3}{2}}(B,\mathm_{|B})}
\end{equation*}
are all summable.

\smallskip

Let $f_1$, $f_2$, $f_3 \in \Lbb_0^1 (A, \mathm)$ be such that
\begin{itemize}
\item $\Sigma_B(f_i) \in \Bcal_1$ for all $i \in \{1,2,3\}$,
\item $ \tau^3 (A, \mathm, T; g_1,g_2,g_3) $ is well defined for any choice of $g_i\in\{f_i,\Sigma_B(f_i)\mathbf{1}_B\}$,
\item $\sigma^2\left(A,\mathm,T;f_i,H\mathbf{1}_B\right)$ and $\sigma^2\left(B,\mathm_B,T_B;\Sigma_B(f_i),H\right)$ are absolutely convergent,
\item $(A,\mathm,T)$ mixes $(g_ig_j\circ T^n,C(g_k))$ for any choice of  $g_\ell \in \{f_\ell,\Sigma_B(f_\ell)\mathbf{1}_B\}$ 
and any choice of $i$, $j$, $k$ such that $\{i,j,k\}=\{1,2,3\}$.
\item $(A,\mathm,T)$ mixes $(H \mathbf{1}_B, C(f_i))$ for all $i \in \{1,2,3\}$.
\end{itemize}

Then $ \tau^3 (B, \mathm_{|B}, T_B ; \Sigma_B (f_1), \Sigma_B (f_2), \Sigma_B (f_3))$ is well defined, and
\begin{align}
 & \tau^3 (A, \mathm, T; f_1,f_2,f_3) 
 = \tau^3 (B, \mathm_{|B}, T_B ; \Sigma_B (f_1), \Sigma_B (f_2),  \Sigma_B (f_3)) \label{invarianceinduction3} \\
 & + \frac 1{2}\sum_{\Alt}\sigma^2_{GK}\left(B,\mathm_{|B},T_B; \Sigma_B(f_1),\Sigma_B(f_2)\right)  \left[\sigma_{GK}^2\left(A,\mathm,T;f_3,H\mathbf{1}_B\right)-\sigma_{GK}^2\left(B,\mathm_B,T_B;\Sigma_B(f_3),H\right) \right]\, .\nonumber
\end{align}
\end{theorem}

\begin{remark}
Assume that $\mathm(B)<\infty$, and take $H=\mathbf{1}_B/\mathm(B)$. Since $\Lcal_B\mathbf{1}_B = \mathbf{1}_B$ 
and $\Sigma_B(f_i)$ has null integral, the term $\sigma^2\left(B,\mathm_B,T_B;\Sigma_B(f_3),H\right)$ 
vanishes in Equation~\eqref{invarianceinduction}.

\smallskip

Assume that $\mathm(A)<\infty$, and take $H=\varphi_B /\mathm(A)$. Assume moreover that $(A, \mathm,T)$ 
mixes $(\Sigma_B(f_i),C(\mathbf{1}_A))$. Then the term $\sigma_{GK}^2\left(A,\mathm,T;f_3,H\mathbf{1}_B\right)$ 
vanishes in Equation~\eqref{invarianceinduction}. This follows from Proposition~\ref{prop:GKCobord3}, 
since $\Sigma_B(\mathbf{1}_A)=\varphi_B$ and $f_i$ has null integral.
\end{remark}

Similarly to the proof of Theorem~\ref{theo:InvarianceGreenKubo}, the scheme of our proof of 
Theorem~\ref{theo:Invariancetau3} consists in the following steps:
\begin{itemize}
\item we show the invariance of $\tau^3$ under addition of a coboundary. Then, 
by Equation~\eqref{eq:Cohomologue} and the mixing assumptions, 
\begin{equation*}
 \tau^3 (A, \mathm, T; f_1,f_2,f_3) 
 = \tau^3 (A, \mathm, T; \Sigma_B(f_1)\mathbf{1}_B,\Sigma_B(f_2)\mathbf{1}_B,\Sigma_B(f_3)\mathbf{1}_B)\, .
\end{equation*}
\item by Proposition~\ref{prop:InvarianceInduction} and Lemma~\ref{lem:Liouville}, for any $p \in [1, \infty]$ 
and any $g \in \Lbb^p (B, \mathm_{|B})$ such that $\sum_{n\ge 0}\norm{\Lcal_B^n g}{\Lbb^p(B,\mathm_{|B})}<\infty$ 
and $\sum_{n\ge 0}\norm{\Lcal^n\left(g\mathbf{1}_B\right)}{\Lbb^p(A,\mathm)}<\infty$, 
there exists a constant $C_0(g)$ such that:
\begin{equation}
 \label{ggg}
 \left(\sum_{n\ge 0}\Lcal^n\left(g \mathbf{1}_B\right)\right)_{|B}
 = \sum_{n\ge 0}\Lcal_B^n\left(g\right)+C_0(g)\, .
\end{equation}
\end{itemize}

\begin{proof}[Proof of Theorem \ref{theo:Invariancetau3}]

Recall that $f_i=\Sigma_B(f_i)\mathbf{1}_B+C(f_i)\circ T-C(f_i)$, 
using the notation $C(\cdot)$ introduced in Equation~\eqref{eq:DefinitionCobord}.
Using successively Lemma~\ref{lem:tau3sigma2} and Proposition~\ref{prop:GKCobord3}, 
the mixing assumption leads to:
\begin{align*}
 \tau^3 (A, \mathm, T; f_1,f_2,f_3)
 & = \tau^3 (A, \mathm, T; f_1,f_2,\Sigma_B(f_3)\mathbf{1}_B)\\
 & = \tau^3 (A, \mathm, T; f_1,\Sigma_B(f_2) \mathbf{1}_B,\Sigma_B(f_3)\mathbf{1}_B)\\
 & = \tau^3 (A, \mathm, T; \Sigma_B(f_1)\mathbf{1}_B,\Sigma_B(f_2) \mathbf{1}_B,\Sigma_B(f_3)\mathbf{1}_B)\, .
\end{align*}
This finishes the first step of the proof.

\smallskip

Let us denote by $a_1$, $a_2$, $a_3$, $a_4$ the successive terms inside the sum $\sum_{\Alt}$ 
in the formula of $\tau^3 (A, \mathm, T; \Sigma_B(f_1)\mathbf{1}_B,\Sigma_B(f_2) \mathbf{1}_B,\Sigma_B(f_3)\mathbf{1}_B)$
given by Lemma~\ref{lem:FormulaTau3}. Notice that
\begin{equation}
 \label{eq:A1}
 a_1 
 := \int_A\prod_{i=1}^3 \left( \Sigma_B(f_i)\mathbf{1}_B \right) \dd \mathm 
 = \int_B\prod_{i=1}^3\Sigma_B(f_i)\dd \mathm\, .
\end{equation}

Applying Equation~\eqref{ggg} with $g=\Sigma_B(f_3)$ and $p=3$, the second term can be written:
\begin{align}
 a_2
 & :=\sum_{n \geq 0} \int_A \Sigma_B(f_1)\mathbf{1}_B\Sigma_B(f_2)\mathbf{1}_B \cdot \Lcal^n (\Sigma_B(f_3)\mathbf{1}_B) \dd \mathm\nonumber\\
 & =\sum_{n \geq 0} \int_B \Sigma_B(f_1)\Sigma_B(f_2) \cdot \Lcal_B^n (\Sigma_B(f_3)) \dd \mathm + C_0(\Sigma_B(f_3)) \int_B \Sigma_B(f_1)\Sigma_B(f_2)\dd \mathm \, . \label{eq:A2}
\end{align}

The third term can be rewritten as follows:
\begin{align}
 a_3
 & :=\sum_{n \geq 0} \int_A \Sigma_B(f_1)\mathbf{1}_B\Lcal^n\left(\Sigma_B(f_2)\mathbf{1}_B \Sigma_B(f_3)\mathbf{1}_B\right) \dd \mathm\nonumber\\
 & =\sum_{n \geq 0} \int_B \Sigma_B(f_1)\Lcal^n\left(\Sigma_B(f_2)\boxtimes \Sigma_B(f_3)\mathbf{1}_B\right) \dd \mathm\nonumber\\
 & \hspace{4em} + \int_B\Sigma_B(f_2)\Sigma_B(f_3)\dd\mathm \cdot \sum_{n\ge 0}\int_B\Sigma_B(f_1)\Lcal^n( H\mathbf{1}_B)\dd\mathm\, .\nonumber
\end{align}
Applying Equation~\eqref{ggg} with $g=\Sigma_B(f_2) \boxtimes\Sigma_B(f_3)$ and $p=3/2$, since $\Sigma_B(f_1)$ has null integral, 
we get:
\begin{align*}
 \sum_{n \geq 0} \int_B \Sigma_B(f_1) & \Lcal^n\left(\Sigma_B(f_2)\boxtimes \Sigma_B(f_3)\mathbf{1}_B\right) \dd \mathm \\
 & =\sum_{n \geq 0} \int_B \Sigma_B(f_1)\Lcal_B^n\left(\Sigma_B(f_2)\boxtimes \Sigma_B(f_3)\right) \dd \mathm \\
 & = \sum_{n \geq 0} \int_B \Sigma_B(f_1)\Lcal_B^n\left(\Sigma_B(f_2) \Sigma_B(f_3)\right) \dd \mathm \\
 & \hspace{4em} - \int_B \Sigma_B(f_2) \Sigma_B(f_3) \dd \mathm \cdot \sum_{n \geq 0} \int_B \Sigma_B(f_1)\Lcal_B^n (H) \dd \mathm \, .
\end{align*}
Moreover, due to Lemma~\ref{lem:GKCobord2}, 
\begin{align*}
 \sum_{n\ge 0}\int_B\Sigma_B(f_1)\Lcal^n (H\mathbf{1}_B)\dd\mathm
 & = \sum_{n\ge 0}\int_Af_1\Lcal^n  ( H\mathbf{1}_B)\dd\mathm + \alpha (H \mathbf{1}_B, C(f_1))-\int_A C(f_1)H \mathbf{1}_B\dd\mathm \\
 & = \sum_{n\ge 0}\int_Af_1\Lcal^n  ( H\mathbf{1}_B)\dd\mathm+\alpha (H \mathbf{1}_B, C(f_1))\, ,
\end{align*}
since $C(f_1)$ is supported on $A\setminus B$. Therefore:

\begin{align}
 a_3 
 & = \sum_{n \geq 0} \int_B \Sigma_B(f_1)\Lcal_B^n\left(\Sigma_B(f_2)\Sigma_B(f_3)\right) \dd \mathm \label{eq:A3}\\
 & \hspace{4em} + \sum_{n\geq 0}\left(\int_Af_1\Lcal^n(H\mathbf{1}_B)\dd\mathm -\int_B\Sigma_B(f_1)\Lcal_B^n(H)\dd\mathm\right)\int_B \Sigma_B(f_2)\Sigma_B(f_3)\dd\mathm \nonumber \\
 & \hspace{4em} + \int_B \Sigma_B(f_2) \Sigma_B(f_3) \dd \mathm \cdot \alpha (H \mathbf{1}_B, C(f_1)) \, .\nonumber
\end{align}

Now, applying Equation~\eqref{ggg} with $g=\Sigma_B(f_3)$ and $p=3/2$,
we rewrite the last term as follows:
\begin{align*}
 a_4
 & := \sum_{n,m \geq 0} \int_A \Sigma_B(f_1)\mathbf{1}_B\Lcal^n\left(\Sigma_B(f_2)\mathbf{1}_B  \Lcal^m (\Sigma_B(f_3)\mathbf{1}_B)\right) \dd \mathm\\
 & =\sum_{n,m \geq 0} \int_B \Sigma_B(f_1)\Lcal^n\left(\Sigma_B(f_2)\mathbf{1}_B  \Lcal_B^m (\Sigma_B(f_3))\right)\dd \mathm \\
 & \hspace{4em} + \sum_{n \geq 0} \int_B \Sigma_B(f_1)\Lcal^n\left(\Sigma_B(f_2)\mathbf{1}_B\right) \dd \mathm \cdot C_0(\Sigma_B(f_3))\\
 & = \sum_{n,m \geq 0} \int_B \Sigma_B(f_1)\Lcal^n\left(\mathbf{1}_B\left(\Sigma_B(f_2) \boxtimes \Lcal_B^m (\Sigma_B(f_3))\right)\right)\dd \mathm\\
 & \hspace{4em} + \sum_{n \geq 0} \int_B \Sigma_B(f_1) \cdot \Lcal^n (H\mathbf{1}_B)\dd \mathm \cdot \sum_{m\geq 0}\int_B \Sigma_B(f_2)\Lcal_B^m (\Sigma_B(f_3))\dd \mathm\\
 & \hspace{4em} + \sum_{n \geq 0} \int_B \Sigma_B(f_1) \cdot \Lcal_B^n \left(\Sigma_B(f_2)\right) \dd \mathm \cdot C_0(\Sigma_B(f_3))\, .
\end{align*}

Applying Equation~\eqref{ggg} with $g=\Sigma_B(f_2)\boxtimes \Lcal_B^m (\Sigma_B(f_3))$ and $p=3/2$, 
and using the fact that $\Sigma_B(f_1)$ has hull integral, we can replace $\Lcal^n(\mathbf{1}_B \ldots)$ by $\Lcal_B^n (\ldots)$ in the above formula:
\begin{align*}
 \sum_{n,m \geq 0} \int_B \Sigma_B(f_1)\Lcal^n & \left(\mathbf{1}_B\left(\Sigma_B(f_2) \boxtimes \Lcal_B^m (\Sigma_B(f_3))\right)\right)\dd \mathm \\
 & = \sum_{n,m \geq 0} \int_B \Sigma_B(f_1)\Lcal_B^n\left(\Sigma_B(f_2) \boxtimes \Lcal_B^m (\Sigma_B(f_3))\right)\dd \mathm \\
 & = \sum_{n,m \geq 0} \int_B \Sigma_B(f_1)\Lcal_B^n\left(\Sigma_B(f_2) \Lcal_B^m (\Sigma_B(f_3))\right)\dd \mathm \\
 & \hspace{4em} - \sum_{n,m \geq 0} \int_B \Sigma_B(f_1) \Lcal_B^n (H) \dd \mathm \cdot \sum_{m\geq 0}\int_B \Sigma_B(f_2)\Lcal_B^m (\Sigma_B(f_3))\dd \mathm \, .
\end{align*}

In addition, $\Sigma_B (f_1) = f_1+C(f_1) \circ T - C(f_1)$, so that, by Lemma~\ref{lem:GKCobord2}:
\begin{align*}
 \sum_{n \geq 0} \int_B \Sigma_B(f_1) \cdot \Lcal^n (H\mathbf{1}_B)\dd \mathm 
 & = \sum_{n \geq 0} \int_A f_1 \cdot \Lcal^n (H\mathbf{1}_B)\dd \mathm +\alpha (C(f_1), H\mathbf{1}_B) - \int_B C(f_1) H\dd \mathm \\
 & = \sum_{n \geq 0} \int_A f_1 \cdot \Lcal^n (H\mathbf{1}_B)\dd \mathm +\alpha (C(f_1), H\mathbf{1}_B) \, , \\
\end{align*}
again using the fact that $C(f_1) \equiv 0$ on $B$. Finally, we get:
\begin{align*}
 a_4
 & = \sum_{n,m \geq 0} \int_B \Sigma_B(f_1)\Lcal_B^n\left(\Sigma_B(f_2) \Lcal_B^m (\Sigma_B(f_3))\right)\dd \mathm \\
 & \hspace{4em} + \sum_{n \geq 0} \left( \int_A f_1 \Lcal^n (H\mathbf{1}_B)\dd \mathm - \int_B \Sigma_B(f_1)\Lcal_B^n (H)\dd \mathm \right) \sum_{m\geq 0}\int_B \Sigma_B(f_2)\Lcal_B^m (\Sigma_B(f_3))\dd \mathm \\
 & \hspace{4em} + \sum_{n \geq 0} \int_B \Sigma_B(f_1) \cdot \Lcal_B^n \left(\Sigma_B(f_2)\right) \dd \mathm \cdot C_0(\Sigma_B(f_3)) \\
 & \hspace{4em} + \alpha (C(f_1), H\mathbf{1}_B) \sum_{m\geq 0}\int_B \Sigma_B(f_2)\Lcal_B^m (\Sigma_B(f_3))\dd \mathm\,.
\end{align*}

Combining this with Equations~\eqref{eq:A1}, \eqref{eq:A2} and~\eqref{eq:A3}, 
with the weights given by Lemma~\ref{lem:FormulaTau3}, yields:
\begin{align}
& \tau^3 (A, \mathm, T; f_1,f_2,f_3) - \tau^3 (B, \mathm_{|B}, T_B ; \Sigma_B (f_1), \Sigma_B (f_2), \Sigma_B (f_3)) \nonumber \\
& = \sum_{\Alt}\left[-\frac{1}{2}\int_B \Sigma_B(f_1)\Sigma_B(f_2)\dd \mathm \cdot C_0(\Sigma_B(f_3)) \right. \label{eq:Tau3Ligne1} \\
& \hspace{4em} -\frac{1}{2} \sum_{n \geq 0} \left(\int_Af_1\Lcal^n(H\mathbf{1}_B)\dd\mathm - \int_B\Sigma_B(f_1)\Lcal_B^n(H)\dd\mathm\right) \int_B\Sigma_B(f_2)\Sigma_B(f_3)\dd\mathm \label{eq:Tau3Ligne2} \\
& \hspace{4em} -\frac{1}{2} \alpha(H \mathbf{1}_B, C(f_1)) \int_B\Sigma_B(f_2)\Sigma_B(f_3)\dd\mathm \label{eq:Tau3Ligne3} \\
& \hspace{4em} + \sum_{n \geq 0} \int_B \Sigma_B(f_1)\Lcal_B^n\left(\Sigma_B(f_2)\right) \dd \mathm \cdot C_0(\Sigma_B(f_3)) \label{eq:Tau3Ligne4} \\
& \hspace{4em} + \sum_{n \geq 0} \left(\int_Af_1\Lcal^n(H\mathbf{1}_B)\dd\mathm -\int_B\Sigma_B(f_1)\Lcal_B^n(H)\dd\mathm\right) \sum_{m\geq 0} \int_B \Sigma_B(f_2)\Lcal_B^m (\Sigma_B(f_3))\dd \mathm \label{eq:Tau3Ligne5} \\
& \hspace{4em} \left.+ \alpha (H\mathbf{1}_B, C(f_1)) \sum_{m\geq 0}\int_B \Sigma_B(f_2)\Lcal_B^m (\Sigma_B(f_3))\dd \mathm\right]\, . \label{eq:Tau3Ligne6}
\end{align}
In the alternated sum above, we put together the lines~\eqref{eq:Tau3Ligne1} and~\eqref{eq:Tau3Ligne4} 
and the permutations $(f_i, f_j, f_k)$ and $(f_j, f_i, f_k)$. We also put together the lines~\eqref{eq:Tau3Ligne2} and~\eqref{eq:Tau3Ligne5} 
and the permutations $(f_i, f_j, f_k)$ and $(f_i, f_k, f_j)$. Finally, we put together the lines~\eqref{eq:Tau3Ligne3} and~\eqref{eq:Tau3Ligne6} 
and the permutations $(f_i, f_j, f_k)$ and $(f_i, f_k, f_j)$. This yields:
\begin{align}
& \tau^3 (A, \mathm, T; f_1,f_2,f_3) - \tau^3 (B, \mathm_{|B}, T_B ; \Sigma_B (f_1), \Sigma_B (f_2),  \Sigma_B (f_3)) \nonumber\\
& = \frac{1}{2} \sum_{\Alt} \sum_{n\geq 0}\left(\int_Af_1\Lcal^n(H\mathbf{1}_B)\dd\mathm -\int_B\Sigma_B(f_1)\Lcal_B^n(H)\dd\mathm\right)  \sigma^2_{GK}(B,\mathm_{|B},T_B; \Sigma_B(f_2),\Sigma_B(f_3)) \nonumber \\
& \hspace{4em} +\frac{1}{2}\sum_{\Alt} \sigma^2_{GK}(B,\mathm_{|B},T_B; \Sigma_B(f_1),\Sigma_B(f_2)) \cdot C_0(\Sigma_B(f_3)) \label{Afinal1} \\
& \hspace{4em} +\frac{1}{2}\sum_{\Alt} \sigma^2_{GK}(B,\mathm_{|B},T_B; \Sigma_B(f_2),\Sigma_B(f_3)) \cdot \alpha(H \mathbf{1}_B, C(f_1)) \, . \nonumber
\end{align}
By Lemma~\ref{lem:GKCobord1}, and using the mixing conditions,
\begin{align*}
 C_0(\Sigma_B(f_3))
 & = \sum_{n\ge 0}\int_B H\Lcal^n\left(\Sigma_B(f_3) \mathbf{1}_B\right)\dd\mathm-\sum_{n\ge 0}\int_B H\Lcal_B^n\left(\Sigma_B(f_3)\right)\dd\mathm\\
 & = \sum_{n \geq 0}  \int_B H \Lcal^n (f_3) \dd \mathm + \int_B H C(f_3)\circ T \dd \mathm \\
 & \hspace{4em} - \alpha(H\mathbf{1}_B, C(f_3)) -\sum_{n\ge 0}\int_B H\Lcal_B^n\left(\Sigma_B(f_3)\right)\dd\mathm\\
 & = \sum_{n \geq 0} \int_B H\Lcal^n (f_3) \dd \mathm +\int_BH \Sigma_B(f_3) \dd \mathm \\
 & \hspace{4em} -\int_BH f_3 \dd \mathm -\alpha(H\mathbf{1}_B, C(f_3)) -\sum_{n\ge 0}\int_B H\Lcal_B^n\left(\Sigma_B(f_3)\right)\dd\mathm\, ,
\end{align*}
since $C(f_3)\circ T=\Sigma_B(f_3)-f_3$ on $B$. This, combined with Equation~\eqref{Afinal1}, leads to Equation~\eqref{invarianceinduction3}.
\end{proof}

\subsection{Invariants for a Bernoulli scheme}

Let us show how the invariants $\sigma^2$ and $\tau^3$ play out on a simple example : 
a Bernoulli scheme.
We shall see that the conclusions of Theorems~\ref{theo:InvarianceGreenKubo} and~\ref{theo:Invariancetau3} hold true
on a simple example  by a straighforward computation.

\smallskip

Let $A := \{0,1\}^\N$ and $T$ be the one-sided shift on A. 
Fix $p \in (0,1)$, and let $\mu := (p \delta_0+(1-p)\delta_1)^{\otimes \N}$. 
Let $f (\omega):= -(1-p) \mathbf{1}_0 (\omega_0) + p \mathbf{1}_1 (\omega_0)$. 
Then $\int_A f \dd \mu = 0$, and the random variables $(f \circ T^n)_{n \geq 0}$ 
are i.i.d..

\smallskip

Let us induce on the set $B = \{0\} \times \{0,1\}^{\N_+}$, with $\mu (B) = p$. Then 
$\Sigma_B (f) = -(1-p)+pG$, where the distribution of 
$G+1$
under $\mu ( \cdot | B)$ 
is geometric of parameter $p$. By the strong Markov property, 
the random variables $(\Sigma_B (f) \circ T_B^n)_{n \geq 0}$ are i.i.d..

\smallskip

By Kac's formula,
\begin{equation*}
 0 
 = \int_A f \dd \mu 
 = \int_B \Sigma_B (f) \dd \mu 
 = \mu (B) \left[p \Ebb (G)-(1-p) \right],
\end{equation*}
so we recover $\Ebb (G) = \frac{1-p}{p}$.

\smallskip

Assuming we can use Theorem~\ref{theo:InvarianceGreenKubo}, 
\begin{align*}
 p(1-p) 
 & = \int_A f^2 \dd \mu 
 = \sigma^2_{GK} (A, \mathm, T; f,f) \\
 & = \sigma^2_{GK} (B, \mathm_{|B}, T_B; \Sigma_B (f),\Sigma_B (f)) 
 = \mu(B) \Ebb [(pG-(1-p))^2],
\end{align*}
which yields the correct identity $\Var (G) = \frac{1-p}{p^2}$.

\smallskip

Finally, assuming we can use Theorem~\ref{theo:Invariancetau3} 
with $H = \varphi_B = 1+G$,
\begin{align*}
 p(1-p)(2p-1)
 & = \int_A f^3 \dd \mu 
 = \tau^3 (A, \mathm, T; f,f,f) \\
 & = \tau^3 (B, \mathm_{|B}, T_B; \Sigma_B (f),\Sigma_B (f), \Sigma_B (f)) \\ 
 & \hspace{2em} - 3 \sigma^2_{GK} (B, \mathm_{|B}, T_B; \Sigma_B (f),\Sigma_B (f)) \sigma^2_{GK} (B, \mathm_{|B}, T_B; \Sigma_B (f),1+G) \\
 & = \mu(B) \Ebb [(pG-(1-p))^3] - 3 \mu(B)^2 \Ebb [(pG-(1-p))^2] \Ebb [(pG-(1-p))(1+G)] \\
 & = p \Ebb [(pG-(1-p))^3] - 3 p (1-p) \Ebb [(pG-(1-p))^2] \\
 & = p \Ebb [(pG-(1-p))^3] - 3 p (1-p)^2,
\end{align*}
which yields the correct identity 
\begin{equation*}
 \Ebb [(G-\Ebb (G))^3]
 = \frac{(1-p)(2-p)}{p^3}.
\end{equation*}

More generally, and up to checking the hypotheses of Theorems~\ref{theo:InvarianceGreenKubo} 
and~\ref{theo:Invariancetau3}, this method yields closed forms for the first $3$ moments of $\Sigma_B (f)$, 
when $f$ is an observable of a finite-state Markov chain with null integral and $B$ is a single state.

\section{The distributional point of view}
\label{sec:Distributions}

The invariance under induction of Green-Kubo's formula, Theorem~\ref{theo:InvarianceGreenKubo}, 
was proved for a small class of dynamical systems preserving an infinite measure in~\cite{PeneThomine:2017}. 
The proof relied on a probabilistic interpretation of this formula, as the asymptotic variance 
in a limit theorem. We now expand on this point of view, and give criterions -- distinct from 
Theorem~\ref{theo:InvarianceGreenKubo} -- to prove the invariance of Green-Kubo's formula.

\smallskip

In this section, we always assume that $(A, \mathm, T)$ is ergodic and recurrent, 
and that $B \subset A$ has positive finite measure.

\subsection{Invariance under induction of limits of Birkhoff sums}

Assume first that $f$ is non-negative. By Hopf's ergodic theorem~\cite[\S$14$, Individueller Ergodensatz f\"ur Abbildungen]{Hopf:1937}, 
$\mathm$-almost surely and $\mathm_{|B}$-almost surely respectively,
\begin{equation}
 \label{eq:Hopf1}
 \lim_{n \to +\infty}\frac{\sum_{k=0}^{n-1} f\circ T^k}{\sum_{k=0}^{n-1} \mathbf{1}_B \circ T^k} 
 = \frac{\int_A f \dd \mathm}{\mathm (B)},
\end{equation}
\begin{equation}
 \label{eq:Hopf2}
 \lim_{n \to +\infty}\frac{\sum_{k=0}^{n-1} \Sigma_B(f) \circ T_B^k}{\sum_{k=0}^{n-1} \mathbf{1}_B \circ T_B^k} 
 = \frac{\int_B \Sigma_B(f) \dd \mathm}{\mathm (B)}.
\end{equation}
For almost every $x \in B$, taking the limit in Equation~\eqref{eq:Hopf1} along the subsequence $(\sum_{k=0}^{n-1} \varphi_B \circ T_B^k (x))_{n \geq 0}$ yields:
\begin{equation*}
 \lim_{n \to +\infty} \frac{\sum_{k=0}^{(\sum_{k=0}^{n-1} \varphi_B \circ T_B^k)-1} f\circ T^k}{\sum_{k=0}^{(\sum_{k=0}^{n-1} \varphi_B \circ T_B^k)-1} \mathbf{1}_B \circ T^k} 
 = \frac{\int_A f \dd \mathm}{\mathm (B)}.
\end{equation*}
Since $\sum_{k=0}^{(\sum_{\ell=0}^{n-1} \varphi_B \circ T_B^\ell)-1} f\circ T^k = \sum_{k=0}^{n-1} \Sigma_B(f) \circ T_B^k$, we get:
\begin{equation*}
 \lim_{n \to +\infty} \frac{\sum_{k=0}^{n-1} \Sigma_B(f) \circ T_B^k}{\sum_{k=0}^{n-1} \mathbf{1}_B \circ T_B^k} 
 = \frac{\int_A f \dd \mathm}{\mathm (B)}.
\end{equation*}
Comparing this limit with that of Equation~\eqref{eq:Hopf2} yields $\int_B \Sigma_B(f) \dd \mathm = \int_A f \dd \mathm$. This holds 
for any positive function $f$. As any integrable function is the difference of two positive integrable functions, we get the 
equality of the integrals for any $f \in \Lbb^1 (A, \mathm)$.

\smallskip

The same result holds true in continuous time, for suspension semi-flows.
Let us consider a  suspension semi-flow  $(\Mcal,\mu,(Y_t)_t)$ over 
a dynamical system $(M,\nu,T)$ with roof function $\varphi: M \to (0,+\infty)$. 
In other words,
\begin{itemize}
 \item $\Mcal =\{(x,s)\in M\times [0,+\infty)\}$ with the identification $(x,s+\varphi(x)) \sim (T(x),s)$,
 \item $\mu=\nu\otimes\Leb$ on the fundamental domain $\{(x,s)\in M\times [0,+\infty) : \ s \leq \varphi (x)\}$,
 \item $Y_t(x,s)=(x,t+s)$ for all $(x,s) \in \Mcal$ and $t \geq 0$.
\end{itemize}
Let $f \in \Lbb^1 (\Mcal, \mu)$. Set $\Sigma_B (f)(x) := \int_0^{\varphi(x)} f(x,s) \dd s$. 
Using the same reasoning (with Hopf's ergodic theorem for semi-flows, and 
with $g (x,s) = \varphi(x)^{-1}$ instead of $\mathbf{1}_B$ in the denominator), 
we get:
\begin{equation*}
 \int_M \Sigma_B(f) \dd \mathm 
 = \int_\Mcal f \dd \mu.
\end{equation*}

As mentioned before, the Green-Kubo formula is a bilinear form which shares with the integral, 
at least in some cases, such an invariance under induction (Theorem~\ref{theo:InvarianceGreenKubo}). 
We now present a distributional argument for this invariance.

\subsection{Invariance under induction of the Green-Kubo formula}
\label{subsec:GreenKuboInvariance}

A short computation (see e.g.\ the proof of \cite[Lemma~A.2]{PeneThomine:2017}) shows that
\begin{equation}
 \label{formulavar1}
 \sigma^2_{GK} (A, \mathm, T;f,f)
 = \lim_{n\rightarrow +\infty} \int_A \left(\frac{1}{\sqrt{n}}\sum_{k=0}^{n-1}f\circ T^k\right)^2 \dd \mathm \, ,
\end{equation}
where the limit in the right hand-side is taken in the Ces\`aro sense. In particular, 
when $\mathm$ is a probability measure and $\left| \int_A f\cdot f\circ T^n \dd\mathm\right|$ 
is summable, $\sigma^2_{GK} (A, \mathm, T;f,f)$ is the limit of the variances of the Birkhoff sums.

\smallskip

For some systems with $\mathm$ infinite, it has been shown in~\cite{PeneThomine:2017} that 
\begin{equation}\label{formulavar2}
 \sigma^2_{GK} (A, \mathm, T;f,f)
 = \lim_{n\rightarrow +\infty} \frac{\int_A \left(\sum_{k=0}^{n-1}f\circ T^k\right)^2\, h \dd \mathm}   {\int_A \left(\sum_{k=0}^{n-1}h\circ T^k\right)\, h \dd \mathm} \, ,
\end{equation}
for some density probability $h$ with respect to $\mathm$.
Note that, when $\mathm$ is a probability measure and $h = \mathbf{1}_A$,
Equation~\eqref{formulavar2} and~\eqref{formulavar1} coincide.

\smallskip

Moreover, as explained in~\cite{Pene:2018}, for a wide family of $\Z^d$-extensions and for suitable null-integral observables,
$\int_A f \cdot f\circ T^n \dd \mathm = O(n^{1-\frac{d}{2}})$ for $d \neq 2$, and 
$O (\ln (n))$ for $d=2$. Hence, in both cases above, $\sigma^2_{GK} (A, \mathm, T ; f)$
is the asymptotic variance of $\left(\frac{1}{\mathfrak{a}_n}\sum_{k=0}^{n-1}f\circ T^k\right)_{n \geq 0}$ 
with respect to some probability measure $h \dd \mathm$ and for some sequence $(\mathfrak{a}_n)_{n \geq 0}$ 
which diverges to $+\infty$.

\smallskip

Assume that $\mathm(A)=1$, and consider a measurable subset $B$ of $A$ of $\mathm$-positive measure. Our goal is to use 
this interpretation of the Green-Kubo formula in order to establish sufficient conditions ensuring that
\begin{equation*}
 \sigma^2_{GK} (A, \mathm, T;f,f)
 = \sigma^2_{GK} (B, \mathm_{|B}, T_B ; \Sigma_B (f), \Sigma_B (f)).
\end{equation*}

Inducing the system speeds up the process $(\sum_{k=0}^{n-1} f\circ T^k)_{n \geq 0}$. 
If both the initial system $(A, \mathm, T)$ and the induced system $(B, \mathm_B, T_B)$ satisfy a central limit 
theorem with respective variances $\sigma^2_{GK} (A, \mathm, T;f,f)$ and $\mathm (B)^{-1} \sigma^2_{GK} (B, \mathm_{|B}, T_B ; \Sigma_B (f), \Sigma_B (f))$, 
and if we can control the time change between the corresponding processes, then we can show 
the invariance under induction of $\sigma_{GK}^2$. The control on the time-change can be achieved 
using an invariance principle (see~\cite[Chapter~14]{Billingsley:1999}).

\begin{proposition}
\label{prop:InvarianceGKDistribution}

Let $(A,\mathm, T)$ be a recurrent and ergodic \emph{probability}-preserving dynamical system. Let $B\subset A$ be a 
measurable set with positive measure. 
Let $f \in \Lbb^2 (A, \mathm)$. Assume that $\sigma^2_{GK} (A, \mathm, T;f,f)$ and 
$\sigma^2_{GK} (B, \mathm_{|B}, T_B ; \Sigma_B (f), \Sigma_B (f))$ both converge in Ces\`aro sense. 

\smallskip

Assume moreover that:
\begin{itemize}
 \item the process $(\frac{1}{\sqrt{n}}(\sum_{k=0}^{\lfloor nt\rfloor-1} \Sigma(f)\circ T_B^k)_{t \geq 0})_{n\geq 0}$ converges 
  in distribution (with respect to $\mathm(\cdot|B)$ and to the metric $\Jcal_1$) to a Brownian motion 
  of variance $\mathm (B)^{-1} \sigma^2_{GK} (B, \mathm_{|B}, T_B ; \Sigma_B (f), \Sigma_B (f))$;
 \item the sequence $(\frac{1}{\sqrt{n}} \sum_{k=0}^{n-1} f\circ T^k)_{n \geq 0}$ converges in distribution to 
   a centered Gaussian random variable with variance $\sigma^2_{GK} (A, \mathm, T;f,f)$.
\end{itemize}
Then $\sigma^2_{GK} (A, \mathm, T;f,f) = \sigma^2_{GK} (B, \mathm_{|B}, T_B ; \Sigma_B (f), \Sigma_B (f))$.
\end{proposition}

Following directly \cite[Theorem 1.1]{MelbourneTorok:2004}, the control on the change of time 
can also be achieved without an invariance principle, as long as the dispersion of 
$ (\sum_{k=0}^{n-1} (\varphi_B-\mathm(B)^{-1})\circ T_B^k)_{n \geq 0}$ is controlled.

\begin{proposition}

Let $a$, $b>0$ such that $\left(1-\frac{1}{a} \right)\left(1-\frac{1}{b} \right) \geq \frac{1}{2}$.
Let $(\Mcal,\mu,(Y_t)_t)$ be the suspension semi-flow over an invertible ergodic probability-preserving 
dynamical system $(B,\mathm,T_B)$ with roof function $\varphi \in \Lbb^a(B, \mathm; \R_+^*)$. 
Let $f\in \Lbb^b(\Mcal, \mu)$ with null integral, and write $\Sigma_B (f) (x) := \int_0^{\varphi(x)} f (Y_s (x,0)) \dd s$. 

\smallskip

Assume morover that:
\begin{itemize}
 \item the sequence $(\frac{1}{\sqrt {n}}\sum_{k=0}^{n-1} \Sigma_B (f)\circ T_B^k)_{n \geq 0}$ converges in distribution (with respect to $\mathm$) 
   to a Gaussian random variable of variance $\sigma^2_B$;
 \item the sequence $\left(\frac{1}{\sqrt{t}}\int_0^t f\circ Y_s \dd s\right)_{t \geq 0}$ converges in distribution (with respect to $\mu(\Mcal)^{-1} \mu$) to 
   a Gaussian random variable of variance $\mu(\Mcal)^{-1}\sigma^2_A$;
 \item the sequence $(\sum_{k=0}^{n-1} (\varphi_B-\mu(\Mcal))\circ T_B^k)_{n \geq 0}$ converges in distribution to a Gaussian random variable.
\end{itemize}
Then $\sigma^2_A=\sigma^2_B$.
\end{proposition}

\subsection{Application to finite state Markov chains}
\label{subsec:MarkovGK}

Let $(X_n)_{n \geq 0}$ be a Markov chain on a finite state space $\Omega$ 
with an irreducible transition kernel $P$. Let $\mu$ be the unique stationary probability 
on $\Omega$ for this Markov chain, and let $\Psi \in \Omega$ be a single site.

\smallskip

Let $f : \Omega \to \R$ be such that $\int_\Omega f \dd \mu = 0$, 
and let $\Sigma_\Psi (f)$ be the sum of $f$ over an excursion from $\Psi$.

\smallskip

By Kac's formula, $\Ebb (\Sigma_\Psi (f)) = 0$. Let us consider the variance of $\Sigma_\Psi (f)$. 
The stochastic process $(\frac{1}{\sqrt{n}}\sum_{k=0}^{n-1}f(X_k))_{n \geq 0}$ satisfies an invariance principle. 
In addition, the tails of the random variable $\Sigma_\Psi (f)$ decay exponentially fast, 
and the excursions are independent. Hence, we can apply Proposition~\ref{prop:InvarianceGKDistribution}:
\begin{align*}
 \Var(\Sigma_\Psi (f)) 
 & = \mu(\Psi) \left( -\int_\Omega f^2 \dd \mu+2\sum_{n \geq 0} \int_\Omega f\cdot (P^*)^n (f) \dd \mu \right) \\
 & = \mu(\Psi) \int_\Omega f \cdot (2(I-P)^{-1}-I) (f) \dd \mu \, ,
\end{align*}
where the sum in the first line converge in Ces\`aro sense, and $(I-P)^{-1}$ is the inverse of $(I-P)$ 
on functions with null expectation.

\subsection{Application to random walks and \texorpdfstring{$\Z^d$-}{Abelian }extensions}
\label{subsec:Extensions}

One of the motivations for this work is the study of hitting probabilities for $\Z^d$-extensions 
of dynamical systems. As these systems are generalisations of random walks, let us recall 
the problem for random walks.

\smallskip

Let $(X_n)_{n \geq 0}$ be a Markov chain on a countable state space $\Omega$, with irreducible transition kernel $P$ 
and invariant measure $\mu$. Assume that $\mu$ is recurrent, and let $\Psi \subset \Omega$ be finite. 
We are interested in the following question:

\emph{Starting from $\omega \in \Psi$, what is the next site in $\Psi$ that the Markov chain hits?}

The answer is a random variable with values in $\Psi$, whose distribution we would like to compute. Doing this 
for any $\omega \in \Psi$ is equivalent to computing $P_\Psi$.

\smallskip

Let $\Bcal (\Psi, \mu) := \C^\Psi$ and $\Bcal_0 (\Psi, \mu) := \{h \in \Bcal (\Psi, \mu: \ \int_\Psi h \dd \mu = 0 \}$. The matrix $P_\Psi$ 
acts on $V$. It preserves constant functions, and since $\mu_{|\Psi}$ is $P_\psi$-invariant, 
it preserves $V_0$. Hence we only need to understand the action of $P_\Psi$ on $V_0$. 
Irreducibility implies that $(I-P_\Psi)$ is invertible on $V_0$.

\smallskip

As noticed at the end of Subsection~\ref{subsec:MarkovToDynamique}, for any $g \in \Bcal_0 (\Psi, \mu)$, there 
exists a bounded function $f : \Omega \to \C$ such that $(I-P) (f) = g$. Hence, one can define a 
potential kernel $\Gamma : \Bcal_0 (\Psi, \mu) \to \Lbb^\infty (\Omega, \mu)$ such that:
\begin{equation*}
 (I-P) \Gamma (g) 
 = g.
\end{equation*}
Let us also define $\Gamma_\Psi := (I-P_\Psi)^{-1} : \Bcal_0 (\Psi, \mu) \to \Bcal_0 (\Psi, \mu)$. 
Then, by Proposition~\ref{prop:InvarianceMarkov}, for all $g$, $h \in V_0$,
\begin{equation*}
 \int_\Psi \Gamma (g) \cdot h \dd \mu 
 = \int_\Psi \Gamma_\Psi (g) \cdot h \dd \mu.
\end{equation*}
If $\Gamma$ is well understood, then one can recover information on $\Gamma_\Psi$, and from there on 
$P_\Psi = I-\Gamma_\Psi^{-1}$. For recurrent and irreducible random walks on $\Z^d$ ($d \in \{1,2\}$), 
the potential kernel $\Gamma$ can be computed or approximated very well using the Fourier transform~\cite{Spitzer:1976}.

\smallskip

Let $(\overline{A}, \overline{\mathm}, \overline{T})$ be an ergodic dynamical system preserving 
the probability measure $\overline{\mathm}$, and let $F : \overline{A} \to \Z^d$ be measurable. Define:
\begin{itemize}
 \item $A := \overline{A} \times \Z^d$;
 \item $T(x,p) := (\overline{T} (x), p+F(x))$ for $(x,p) \in A$;
 \item $\mathm := \overline{\mathm} \otimes \Leb_{\Z^d}$.
\end{itemize}
Then $(A, \mathm, T)$ is measure-preserving. To make discussions easier, for $\Psi \subset \Z^d$ 
we write $[\Psi] := \overline{A} \times \Psi$. Under the hypothesis that it is ergodic and recurrent, 
one may choose $\Psi \subset \Z^d$ finite, and ask:

\emph{Starting from $(x, p) \in [\Psi]$, what is the next site in $[\Psi]$ 
that the dynamical system hits?}

As with Markov chains, we thus would like to compute the matrix $P_\Psi$ whose entries are the probabilities that, 
starting from $([\{p\}], \overline{\mathm})$, the orbit next hits $[\{p'\}]$. 
One option to compute $P_\Psi$ is to use the transfer operator $\Lcal$. 
For nice enough systems, one can find a good Banach space $\Bcal \subset \Lbb^\infty ([\Psi], \mathm)$, 
with $\Bcal_0 = \{h \in \Bcal : \ \int_{[\Psi]} h \dd \mathm = 0\}$, and such that:
\begin{itemize}
 \item constant functions are in $\Bcal$;
 \item $\Bcal$ is dense in $\Lbb^2 ([\Psi], \mathm)$;
 \item $I-\Lcal_{[\Psi]}$ is invertible from $\Bcal_0$ to itself (with inverse $\Gamma_{[\Psi]}$).
\end{itemize}
Then $I-\Lcal$ is invertible from $\Bcal_0$ to $\Lbb^\infty (A, \mathm)$, by Lemma~\ref{lem:PoissonExtension}; 
denote by $\Gamma$ one of its inverses. It follows that, for all $f$ and $g$ in $\Bcal_0$:
\begin{equation*}
 \int_{[\Psi]} \Gamma (g) \cdot h \dd \mathm 
 = \int_{[\Psi]} \Gamma_{[\Psi]} (g) \cdot h \dd \mathm.
\end{equation*}
If $\Gamma$ is well understood, then one can recover $\Gamma_{[\Psi]}$, and from there 
$\Lcal_{[\Psi]}$. Under spectral assumptions, the potential kernel $\Gamma$ can again be computed 
or approximated very well using the Fourier transform (see e.g.~\cite[Proposition~1.6]{PeneThomine:2017} 
for a weaker variant).

\smallskip

Note that, if one is only interested in the transition probabilities, then one only needs 
to compute $\int_{[\Psi]} \Lcal_{[\Psi]} (g) \cdot h \dd \mathm$ 
for $g$, $h$ constant on each $[\{p\}]$. However, the 
set of functions constant on each $[\{p\}]$ is typically\footnote{Except in the 
case of true random walks, by the construction of Subsection~\ref{subsec:DynamiqueToMarkov}.} 
not stable under $\Lcal_{[\Psi]}$, which motivates the use of a larger Banach space.

\smallskip

Finally, let us comment a result on hitting probabilities:~\cite[Corollary~1.9]{PeneThomine:2017}. 
It was proved therein that, under technical assumptions and when $(\overline{A}, \overline{\mathm}, \overline{T})$ 
is Gibbs-Markov:
\begin{align*}
 \overline{\mathm} & \left(\text{trajectory starting from } [\{0\}] \text { hits } [\{p\}] \text{ before } [\{0\}] \right) \\
 & \hspace{2em} \sim_{p \to \infty} \sum_{n \geq 0} \left[ 2 \overline{\mathm} (S_n F = 0) - \overline{\mathm} (S_n F = p) - \overline{\mathm} (S_n F = -p)\right].
\end{align*}
This result can be seen as an asymptotic development of $P_{[\{0,p\}]}$ when $p$ goes to infinity. 
The proof uses induction on $[\{0, p\}]$, and then the invariance of the Green-Kubo formula, 
instead of the invariance of the solutions of the Poisson equation. The method of the proof was similar 
in spirit to that of Proposition~\ref{prop:InvarianceGKDistribution}, with additional difficulties 
due to the infinite measure setting.

\smallskip

When looking for a stochastic transition matrix on $n \geq 2$ sites with prescribed 
invariant measure, one has $(n-1)^2$ degrees of freedom. The invariance of the Green-Kubo formula 
gives access only to the symmetrized potential kernel $\Gamma_\Psi+\Gamma_\Psi^*$, 
which yields $n(n-1)/2$ independent constraints. These are enough to recover $\Gamma_\Psi$, 
and then $P_\Psi$, if and only if $n \leq 2$. This explains both the success of the approach in~\cite{PeneThomine:2017}, 
where $|\Psi| = 2$, and the necessity of more accurate tools if one wish to 
work with more than two sites.

\end{document}